\UseRawInputEncoding
\documentclass[12pt, reqno]{amsart}
\usepackage[margin=1in]{geometry}
\usepackage{amssymb,latexsym,amsmath,amscd,amsfonts}
\usepackage{latexsym}
\usepackage[mathscr]{eucal}
\usepackage{bm}
\usepackage{mathptmx}
\usepackage{amssymb}
\usepackage{amsthm}
\usepackage{dcolumn}
\usepackage[all]{xy}
\usepackage{enumitem}
\usepackage[utf8]{inputenc}
\usepackage[dvipsnames]{xcolor}

\def \qed {\hfill \vrule height6pt width 6pt depth 0pt}
\def\textmatrix#1&#2\\#3&#4\\{\bigl({#1 \atop #3}\ {#2 \atop #4}\bigr)}
\def\dispmatrix#1&#2\\#3&#4\\{\left({#1 \atop #3}\ {#2 \atop #4}\right)}
\newcommand{\beg}{\begin{equation}}
	\newcommand{\eeg}{\end{equation}}
\newcommand{\ben}{\begin{eqnarray*}}
	\newcommand{\een}{\end{eqnarray*}}

\newtheorem{thm}{Theorem}[section]
\newtheorem{cor}[thm]{Corollary}
\newtheorem{lem}[thm]{Lemma}

\newtheorem{prop}[thm]{Proposition}
\numberwithin{equation}{section} \theoremstyle{definition}
\newtheorem{defn}[thm]{Definition}
\newtheorem{rem}[thm]{Remark}

\newtheorem{eg}[thm]{Example}

\newcommand{\C}{\mathbb{C}}
\newcommand{\B}{\mathbb{B}}
\newcommand{\G}{\mathbb{G}}
\newcommand{\D}{\mathbb{D}}
\newcommand{\T}{\mathbb{T}}
\newcommand{\N}{\mathbb{N}}
\newcommand{\Z}{\mathbb{Z}}
\newcommand{\Pe}{\mathbb{P}}
\newcommand{\HS}{\mathcal{H}}
\newcommand{\PC}{\overline{\mathbb{P}}}
\newcommand{\BC}{\overline{\mathbb{B}}_2}
\newcommand{\DC}{\overline{\mathbb{D}}}
\newcommand{\UT}{\underline{T}}
\newcommand{\BNC}{\overline{\mathbb{B}}_n}

\def\textmatrix#1&#2\\#3&#4\\{\bigl({#1 \atop #3}\ {#2 \atop #4}\bigr)}
\def\dispmatrix#1&#2\\#3&#4\\{\left({#1 \atop #3}\ {#2 \atop #4}\right)}

\title[Operators associated with the pentablock]{Operators associated with the pentablock and their relations with biball and symmetrized bidisc}
\author[Pal and Tomar]{Sourav Pal and Nitin Tomar}

\address[Sourav Pal]{Mathematics Department, Indian Institute of Technology Bombay,
	Powai, Mumbai - 400076, India.} \email{souravmaths@gmail.com, sourav@math.iitb.ac.in}

\address[Nitin Tomar]{Mathematics Department, Indian Institute of Technology Bombay, Powai, Mumbai-400076, India.} \email{tnitin@math.iitb.ac.in, tomarnitin414@gmail.com}		



\begin{document}

	\begin{abstract}
	
	A commuting triple of Hilbert space operators $(A,S,P)$ is said to be a \textit{$\mathbb{P}$-contraction} if the closed pentablock $\overline{\mathbb P}$ is a spectral set for $(A,S,P)$, where 
		\[
		\mathbb{P}:=\left\{(a_{21}, \mbox{tr}(A_0), \mbox{det}(A_0))\ : \ A_0=[a_{ij}]_{2 \times 2} \; \; \& \;\; \|A_0\| <1 \right\} \subseteq \mathbb{C}^3.
		\]
		 A commuting triple of normal operators $(A, S, P)$ acting on a Hilbert space is said to be a \textit{$\mathbb P$-unitary} if the joint spectrum $\sigma_T(A, S, P)$ of $(A, S, P)$ is contained in the distinguished boundary $b\mathbb{P}$ of $\PC$. Also, $(A, S , P)$ is called a \textit{$\mathbb P$-isometry} if it is the restriction of a $\mathbb P$-unitary $(\hat A, \hat S, \hat P)$ to a joint invariant subspace of $\hat A, \hat S, \hat P$. We find several characterizations for the $\mathbb P$-unitaries and $\mathbb P$-isometries. We show that every $\mathbb P$-isometry admits a Wold type decomposition that splits it into a direct sum of a $\mathbb P$-unitary and a pure $\mathbb P$-isometry. Moving one step ahead we show that every $\mathbb P$-contraction $(A,S,P)$ possesses a canonical decomposition that orthogonally decomposes $(A,S,P)$ into a $\mathbb P$-unitary and a completely non-unitary $\mathbb P$-contraction. We find a necessary and sufficient condition such that a $\mathbb P$-contraction $(A, S, P)$ dilates to a $\mathbb P$-isometry $(X, T, V)$ with $V$ being the minimal isometric dilation of $P$. Then we show an explicit construction of such a conditional dilation. We show interplay between operator theory on the following three domains: the pentablock, the biball and the symmetrized bidisc.		
	\end{abstract} 
	
	\maketitle

	\tableofcontents
	
		\section{Introduction} \label{Intro}
	
	\vspace{0.2cm}
	
\noindent Throughout the paper, all operators are bounded linear operators acting on complex Hilbert spaces and the algebra of bounded linear operators acting on a Hilbert space $\HS$ is denoted by $\mathcal{B}(\HS)$. For a commuting tuple of operators $\underline{T}=(T_1, \dotsc, T_n)$ acting on a Hilbert space $\HS$, we denote by $\sigma_T(T_1, \dotsc, T_n)$ the polynomial joint spectrum (or simply the joint spectrum) of $(T_1, \dotsc, T_n)$ relative to the closed algebra $\mathcal{A}$ of $\mathcal{B}(\mathcal{H})$ generated by $T_1, \dotsc, T_n$ and the identity operator $I_\HS$ on $\HS$, i.e., 
\[
\sigma_T(T_1, \dotsc, T_n)=\{(\lambda_1, \dotsc, \lambda_n) \in \C^n : I_\HS \notin (T_1-\lambda_1)\mathcal{A}+\dotsc+(T_n-\lambda_n)\mathcal{A} \}.
\]
In this article, we introduce operator theory on the pentablock $\mathbb P$, a domain related to a special case of $\mu$-synthesis, which is defined by
\[
		\mathbb{P}:=\left\{(a_{21}, \mbox{tr}(A_0), \mbox{det}(A_0))\ : \ A_0=[a_{ij}]_{2 \times 2}, \|A_0\| <1 \right\} \subseteq \mathbb{C}^3.
		\]	
	 Also, we study operators having the closed Euclidean unit ball $\overline{\B}_n$ as a spectral set and then connect the operator theory of the three domains: $\Pe, \B_2$ and the symmetrized bidisc $\mathbb G_2$, where
	\begin{align*}
	\mathbb B_n & =\{ (w_1,\dots , w_n)\in \C^n \; : \, |w_1|^2+ \dots + |w_n|^2 <1 \}, \\
	\mathbb G_2 & = \{ (z_1+z_2,z_1z_2) \in \C^2 \;: \; |z_1|<1, \, |z_2|<1 \}.	
	\end{align*}
In \cite{Agler}, Agler, Lykova and Young introduced the pentablock to study a special case of $\mu$-synthesis.  The $\mu$-synthesis is a part of the theory of robust control of
systems comprising interconnected electronic devices whose outputs
are linearly dependent on the inputs. Given a linear subspace $E$
of $\mathcal M_n(\mathbb C)$, the space of all $n \times n$
complex matrices, the functional
\[
\mu_E(A_0):= (\text{inf} \{ \|Y \|: Y\in E \text{ and } (I-A_0Y)
\text{ is singular } \})^{-1}, \; A_0\in \mathcal M_n(\mathbb C),
\]
is called a \textit{structured singular value}, where the linear
subspace $E$ is referred to as the \textit{structure}. If
$E=\mathcal M_n(\mathbb C)$, then $\mu_E (A_0)$ is equal to the
operator norm $\|A_0\|$, while if $E$ is the space of all scalar multiples of the identity matrix, then $\mu_E(A_0)$
is the spectral radius $r(A_0)$. For any linear subspace $E$ of
$\mathcal M_n(\mathbb C)$ that contains the identity matrix $I$,
$r(A_0)\leq \mu_E(A_0) \leq \|A_0\|$. We refer to the pioneering
work of Doyle \cite{Doyle} on control-theory and motivation
behind considering $\mu_E$. Also, for further details on this topic an interested reader can
see \cite{Francis}. Given distinct points $\alpha_1, \dots , \alpha_d \in \D$, the open unit disk in the complex-plane $\C$, and matrices $B_1, \dots , B_d \in \mathcal M_n(\C)$, the aim of $\mu$-synthesis is to find an
analytic function $F:\,\D \rightarrow \mathcal M_n(\C)$ with $\mu_E(F(\lambda))<1$ for all $\lambda \in \mathbb D$ such that $F$ interpolates the given data, i.e. $F(\alpha_i)=B_i$ for $1\leq i \leq d$. The pentablock arises naturally in connection with a special case of $\mu$-synthesis. Indeed, if
\[
E=\bigg\{ \begin{pmatrix}
\lambda & \mu \\
0 & \lambda
\end{pmatrix}\,: \, \lambda , \mu \in \C \bigg\},
\]
then $\mu_E(A_0)<1$ for $A_0=[a_{ij}]_{2 \times 2}$ if and only if $(a_{21}, \mbox{tr}(A_0), \mbox{det}(A_0)) \in \Pe$. Thus, the function
\[
\pi: \,A_0=[a_{ij}] \mapsto (a_{21}, \mbox{tr}(A_0), \mbox{det}(A_0))
\]
maps the $\mu_E$ unit ball onto the pentablock. However, Agler, Lykova and Young refined this result in \cite{Agler} and showed that the pentablock is the image under $\pi$ of the norm unit ball $\B_{\|.\|}$ in $\mathcal M_2(\C)$, which is strictly smaller than the $\mu_E$ unit ball. The pentablock has attracted considerable attentions in recent past from complex geometric and function theoretic aspects \cite{Alsheri, Jindal, Kosinski, Su, SuII, Zapalowski}. In this article, we initiate operator theory on the pentablock. Thus, our primary object of study in this paper is a commuting operator triple that has the closed pentablock as a spectral set.	
	
\begin{defn}
Let $X \subseteq \C^n$ be a polynomially convex compact set. Then $X$ is said to be a \textit{spectral set} for a commuting tuple of operators $(T_1, \dotsc, T_n)$ if von Neumann's inequality holds for every polynomial $p \in \C[z_1, \dots , z_n]$, that is,
\begin{equation}\label{eqn_spec}
	\|p(T_1, \dotsc, T_n) \| \leq \sup \{|p(z_1, \dotsc, z_n)| : (z_1, \dotsc, z_n)\in X\} = \|p\|_{\infty, \, X \,}.
\end{equation}	
Furthermore, $X$ is said to be a \textit{complete spectral set} for $(T_1, \dotsc, T_n)$ if
\[
\|f(T_1, \dotsc, T_n) \| \leq \sup \{\|f(z_1, \dotsc, z_n)\| : (z_1, \dotsc, z_n)\in X\} 
\]	
holds for every matricial polynomial $f=[f_{ij}]$, where each $f_{ij} \in \C[z_1, \dots , z_n]$. Let $\Omega \subseteq \mathbb{C}^n$ be a bounded domain such that $\overline{\Omega}$ is polynomially convex. A commuting $n$-tuple of operators $(T_1,\dotsc,T_n)$ is said to be an \textit{$\Omega$-contraction} (or, \textit{$\overline{\Omega}$-contraction}) if $\overline{\Omega}$ is a spectral set for $(T_1,\dotsc,T_n)$.
\end{defn}
Unitaries, isometries and co-isometries are special classes of contractions. A unitary is a normal operator having its spectrum on the unit circle $\T$. An isometry is the restriction of a unitary to an invariant subspace and a co-isometry is the adjoint of an isometry. In an analogous manner, we define unitary, isometry and co-isometry associated with the pentablock. To do so, we briefly describe the definition of distinguished boundary. For a bounded domain $\Omega  \subset \C^n$, the \textit{distinguished boundary} of $\overline{\Omega}$ is the smallest closed subset $b\Omega$ (or, $b\overline{\Omega}$) of $\overline{\Omega}$ such that every function that is analytic in $\Omega$ and continuous on $\overline{\Omega}$ attains its maximum modulus on $b \Omega$.

\begin{defn}\label{defn:P-unitary}
 Let $A,S,P$ be commuting operators acting on a Hilbert space $\HS$. Then the triple $(A,S,P)$ is called 
 \begin{itemize}
 
 \item[(i)] a $\Pe$-\textit{unitary} if $A,S,P$ are normal operators and the polynomial joint spectrum $\sigma_T(A,S,P)$ lies in the distinguished boundary of the pentablock ;
  
 \item[(ii)] a $\Pe$-\textit{isometry} if there is a Hilbert space $\mathcal K \supseteq \HS$ and a $\Pe$-unitary $(\hat A, \hat S, \hat P)$ on $\mathcal K$ such that $\HS$ is a joint invariant subspace of $\hat A, \hat S, \hat P$ and $\hat A|_{\HS}=A, \hat S|_{\HS}= S$ and $\hat P|_{\HS}=P$ ;
 
 \item[(iii)] a $\Pe$-\textit{co-isometry} if $(A^*,S^*,P^*)$ is a $\Pe$-isometry;
 
 \item[(iv)] a \textit{completely non-unitary $\Pe$-contraction} or simply a \textit{c.n.u. $\Pe$-contraction} if $(A, S, P)$ is a $\Pe$-contraction and there is no closed joint reducing subspace of $A, S, P$ restricted to which $(A, S, P)$ becomes a $\Pe$-unitary.
 \end{itemize}
Similarly, one defines unitary, isometry and co-isometry for the classes of $\overline{\B}_n$-contractions and $\Gamma$-contractions. Moreover, an isometry (on $\HS$) associated with a domain is called \textit{pure} if there is no non-zero joint reducing subspace of the isometry on which it acts like a unitary associated with the domain.
\end{defn}

 \begin{rem}
 	One can define unitary with respect to a bounded domain whose closure is polynomially convex (e.g., $\Pe$-unitary as in Definition \ref{defn:P-unitary} for the domain pentablock) by considering the Taylor joint spectrum instead of the polynomial joint spectrum in the definition. These two notions of joint spectrum coincide in the context of this paper as we deal here with domains having polynomially convex closures, such as pentablock, Euclidean unit ball and symmetrized bidisc. Thus, the slightly delicate issues surrounding various notions of joint spectrum are not relevant to this paper.
\end{rem}
Amongst the central theorems that constitute the foundation of the one-variable operator theory, the following two results are remarkable: the first is due to von Neumann \cite{vN}, which states that an operator is a contraction if and only if the closed unit disc $\overline{\D}$ is a spectral set for it and the second is Sz.-Nagy's celebrated theorem \cite{Nagy}, which asserts that an operator $T$ on a Hilbert space $\HS$ is a contraction if and only if it dilates to a unitary $U$ acting on a Hilbert space $\mathcal K \supseteq \HS$, i.e. $T^n = P_{\HS}U^n|_{\HS}$ for every positive integer $n$, where $P_{\HS}$ is the orthogonal projection of $\mathcal K$ onto $\HS$. Moreover, such a dilation is called \textit{minimal} if 
\[
\mathcal K = \overline{span}\; \{ U^nh\,:\, h \in \HS, \, n \in \Z \}.
\]
 Thus, von Neumann's famous result compels to realize a contraction as an operator having $\overline{\D}$ as a spectral set and Sz.-Nagy's dilation theorem paves a way to dilate such an operator to a normal operator having its spectrum on the boundary of $\overline{\D}$. Taking cue from such inspiring classical concepts, one considers operators having other domains as spectral sets and in the same spirit studies if they dilate to normal operators associated with the boundary of the domain.
 
 \medskip 
 
  In this article, we first focus on $\Pe$-unitaries and $\Pe$-isometries; characterize them in several different ways and decipher their structures in Sections \ref{P-uni} \& \ref{P-iso}. We show that every $\Pe$-isometry admits a Wold decomposition that splits it into two orthogonal parts of which one is a $\Pe$-unitary and the other is a pure $\Pe$-isometry. This is parallel to the Wold decomposition of an isometry into a unitary and a pure isometry. Also, more generally every contraction orthogonally decomposes into a unitary and a completely non-unitary contraction. A completely non-unitary contraction is a contraction that does not have a unitary part. Such a decomposition is called the \textit{canonical decomposition} of a contraction, see Chapter I of \cite{Nagy} for details. In Section \ref{decomp}, we show that such a canonical decomposition is possible for a $\Pe$-contraction. In Section \ref{Polyball}, we discuss analogues of these results for $\B_n$-contractions.
  
  \smallskip
  
  Operators having the (closed) pentablock as a spectral set have close connections with the operators associated with the biball and the symmetrized bidisc. Indeed, in Section \ref{Prelims} we show that if $(A,S,P)$ is a $\Pe$-contraction then $(A, S \slash 2)$ is a $\B_2$-contraction and $(S,P)$ is a $\Gamma$-contraction. However, a converse to this result does not hold and we show it by a counter example. Naturally, operators associated with the symmetrized bidisc come into the picture while studying $\Pe$-contractions. Operator theoretic aspects of the symmetrized bidisc have rich literature, e.g. \cite{AglerII, AglerVII, Bhattacharyya, Tirtha-Sourav1}. An interested reader may also see the references therein. In \cite{AglerII}, Agler and Young profoundly established the success of rational dilation on the symmetrized bidisc and in \cite{Bhattacharyya}, Bhattacharyya, Pal and Shyam Roy explicitly constructed such a dilation. In Section \ref{dilation}, we mention this dilation theorem. Note that it suffices to find an isometric dilation to a commuting operator tuple associated with a domain, because, by definition every isometry associated with a domain extends to a unitary with respect to the same domain. The explicit $\Gamma$-isometric dilation of a $\Gamma$-contraction from \cite{Bhattacharyya} motivates us to construct explicitly an isometric dilation for a $\Pe$-contraction under certain conditions. Still it is unknown if every $\Pe$-contraction dilates to a $\Pe$-isometry. The fact that the component $P$ of a $\Pe$-contraction $(A,S,P)$ is a contraction leads to the possibility of a $\Pe$-isometric dilation $(X,T,V)$ of $(A,S,P)$, when $V$ is the minimal isometric dilation of $P$. We capitalize this idea in Section \ref{dilation}. In Theorem \ref{thm:main-dilation}, we find a necessary and sufficient condition such that a $\Pe$-contraction $(A,S,P)$ dilates to a pentablock isometry $(X,T,V)$, where $V$ is the minimal isometric dilation of $P$. Then we explicitly construct such a conditional $\Pe$-isometric dilation in Theorem \ref{lem6.9}. The following operator equation in $Z$ associated with a $\Pe$-contraction $(A,S,P)$ plays an important role in these dilation theorems:
  \[
  I-A^*A-\frac{1}{4}S^*S=D_P\big(Z^*Z+\frac{1}{4}FF^*\big)D_P,
  \]
  where $D_P=(I-P^*P)^{1\slash 2}$ and $F$ satisfies $S-S^*P=D_PFD_P$. In Section \ref{dilation}, we also find a necessary and sufficient condition such that the above operator equation has a solution and prove that such a solution, when exists, is unique. At every stage of this paper, we explore and find interaction of a $\Pe$-contraction with $\B_2$-contractions and $\Gamma$-contractions.
  
  \medskip
  
  \noindent \textbf{Note.} After several months of writing this paper and posting to arXiv-math, the article \cite{JindalII} appeared in arXiv-math that has intersection with parts of Theorem \ref{P_unitary} and Theorem \ref{P_isometry} of our paper.
  
  \vspace{0.1cm}

	\section{The $\Pe$-contractions, $\B_2$-contractions and $\Gamma$-contractions}\label{Prelims}
	
	\vspace{0.2cm}
	
	\noindent Recall that a $\Pe$-contraction is a commuting operator triple that has the closed pentablock $\PC$ as a spectral set. Similarly, for a commuting operator pair if the closed biball $\BC$ or the closed symmetrized bidisc $\Gamma$ is a spectral set, then it is called a $\B_2$-contraction or a $\Gamma$-contraction respectively. In this Section, we present a few basic results on $\Pe$-contractions and explore their interactions with $\B_2$-contractions and $\Gamma$-contractions. We begin with an elementary proposition whose proof is a consequence of spectral mapping theorem.	
	\begin{prop}\label{basicprop:02}
	A compact subset $X$ of $\C^n$ is a spectral set for a commuting tuple of normal operators $\underline{N}=(N_1, \dots , N_n)$ if and only if $\sigma_T(N_1, \dots , N_n) \subseteq X$.	
	\end{prop}	
The following result comes naturally as a consequence of the definition of a $\Pe$-contraction.	
\begin{lem} \label{basiclem:01}
Let $(A,S,P)$ on a Hilbert space $\HS$ be a $\Pe$-contraction. Then
\begin{itemize}

\item[(i)] $(A^*,S^*,P^*)$ is a $\Pe$-contraction;

\item[(ii)] $(A|_{\mathcal{L}}, S|_{\mathcal{L}} , P|_{\mathcal{L}})$ is a $\Pe$-contraction for any joint invariant subspace $\mathcal L \subseteq \HS$ of $A,S,P$.

\end{itemize}

\end{lem}	
	
\begin{proof}

$(i)$	Given a polynomial $f(z_1, z_2, z_3)=\underset{0 \leq i,j,k \leq m}{\sum}a_{ijk}z_1^iz_2^jz_3^k$, we define another polynomial 
			\[
			\hat{f}(z_1, z_2, z_3)=\underset{0 \leq i,j,k \leq m}{\sum}\overline{a_{ijk}}z_1^iz_2^jz_3^k.
			\]
			For any $\mathbb{P}$-contraction $(A, S, P)$, we have
			\[
			\|f(A^*, S^*, P^*)\| =\|\hat{f}(A, S, P)^*\|=\|\hat{f}(A, S, P)\| \leq \|\hat{f}\|_{\infty,\; \overline{\mathbb{P}}}.
			\]
			Let $(a, s, p) \in \overline{\mathbb{P}}$ be a point at which $\hat f$ attains its maximum modulus value. Since for every $(z_1, z_2, z_3) \in \overline{\mathbb{P}}$, the conjugate triple $(\overline{z_1}, \overline{z_2}, \overline{z_3}) \in \overline{\mathbb{P}}$, we get that 
			\[
			\|\hat f\|_{\infty, \overline{\mathbb{P}}}=|\hat f(a, s, p)|=|f(\overline{a}, \overline{s}, \overline{p})| \leq \|f\|_{\infty,\; \overline{\mathbb{P}}}.
			\]
			Consequently, it follows that 	
			\[
			\|f(A^*, S^*, P^*)\| \leq \|f\|_{\infty,\; \overline{\mathbb{P}}},
			\]
			for every polynomial $f$ in three variables. Hence, $(A^*, S^*, P^*)$ is a $\mathbb{P}$-contraction if $(A, S, P)$ is a $\mathbb{P}$-contraction.
			
			\medskip

\noindent $(ii)$. Let $\mathcal{L}$ be a joint invariant subspace of a $\mathbb{P}$-contraction $(A, S, P)$ and let $f$ be any polynomial in three variables. Then 
			\[
			\|f(A|_{\mathcal{L}}, S|_{\mathcal{L}} , P|_{\mathcal{L}})\|= \|f(A, S, P)|_{\mathcal{L}}\| \leq \|f(A, S, P)\| \leq  \|f\|_{\infty,\; \overline{\mathbb{P}}}
			\]
and the proof is complete.			
					\end{proof}

Now we move forward to explore relations of $\Pe$ with the biball $\B_2$ and the symmetrized bidisc $\G_2$, which result in interplay between $\Pe$-contractions, $\B_2$-contractions and $\Gamma$-contractions. We start with a couple of useful results from the literature, of which the first is due to Agler and Young \cite{AglerVII}.
	
	\begin{thm}[\cite{AglerVII}, Theorems 2.2 \& 2.6]\label{Gamma_uni}
		Let $S, P$ be commuting operators on a Hilbert space $\mathcal{H}$. Then 
		\begin{enumerate}
			\item $(S, P)$ is a $\Gamma$-unitary if and only if $S=S^*P$, $P$ is unitary and $\|S\| \leq 2$.
			\item $(S, P)$ is a $\Gamma$-isometry if and only if $S=S^*P$, $P$ is isometry and $\|S\| \leq 2$.
		\end{enumerate}
	\end{thm}
	
The next result is due to Agler, Lykova and Young from \cite{Agler} which characterizes the points in $\PC$ in several ways. 
	
	\begin{thm}[\cite{Agler}, Theorem 5.3] \label{thm2.1}
		Let 
		$
		(s,p)=(\beta+\overline{\beta}p, p)=(\lambda_1+\lambda_2, \lambda_1\lambda_2) \in \Gamma \,,
		$
		where $\lambda_1, \lambda_2 \in \overline{\mathbb{D}}$ and $|\beta| \leq 1$. If $|p|=1$, then $\beta=\frac{1}{2}s$. Let $a \in \mathbb{C}$. The following statements are equivalent:
		\begin{enumerate}
			\item $(a, s, p) \in \overline{\mathbb{P}}$;
			
			\vspace{0.1cm} 
			
			\item $|a| \leq \bigg|1-\frac{\frac{1}{2}s\overline{\beta}}{1+\sqrt{1-|\beta|^2}} \bigg|$;
			
			\vspace{0.1cm} 
			
			\item $|a| \leq  \frac{1}{2}|1-\overline{\lambda}_2\lambda_1|+\frac{1}{2}(1-|\lambda_1|^2)^{\frac{1}{2}}(1-|\lambda_2|^2)^{\frac{1}{2}}$;
			
			\vspace{0.1cm} 
			
			\item there exists $A_0 \in M_2(\mathbb{C})$ such that $\|A_0\| \leq 1$ and $\pi(A_0)=(a, s, p)$.
		\end{enumerate}	
	\end{thm}

	\smallskip
	
	It is evident from the above theorem that if $(a, s, p)$ is an element in $\overline{\mathbb{P}}$, then $(s, p) \in \Gamma$. Also, the same holds if we consider $\Pe$ and $\G_2$ (see \cite{Agler}). Indeed, if $(a, s, p) \in \overline{\mathbb{P}}$, then there is a $2 \times 2$ contraction $A_0=[a_{ij}]$ such that	$
	a_{21}=a, tr(A_0)=s \text{ and }det(A_0)=p$. Therefore, $(s,p) \in \Gamma$ as we have from \cite{AglerII} that	$
	\Gamma=\{(tr(A_0), det(A_0)) \in \mathbb{C}^2 \ : \ A_0 \in M_2(\mathbb{C}), \ \|A_0\| \leq 1 \}$. So, we have the following lemma which can also be found in \cite{Agler}.
	
	\begin{lem} \label{basiclem:02}
	If $(a,s,p) \in \PC$ (or $\in \Pe$), then $(s,p) \in \Gamma$ (or $\in \G_2$).
	\end{lem}
	
This scalar result naturally extends to the following operator theoretic version.	
	
		\begin{prop}\label{lem2.3}
		If $(A, S, P)$ is a $\mathbb{P}$-contraction, then $(S, P)$ is a $\Gamma$-contraction.
	\end{prop}
	
	\begin{proof}
		Let $(a, s, p) \in \PC$ be any point. By Theorem \ref{thm2.1}, there is a $2 \times 2$ matrix $A_0=[a_{ij}]$ with $\|A_0\| \leq 1$ such that $\pi(A_0)=(a, s, p)$. Evidently, $|a|=|a_{21}| \leq 1$ and $(s, p)=(tr(A_0), det(A_0))$. Thus, $a \in \DC$ and $(s, p) \in \Gamma$. Therefore, we have that $\PC \subseteq \DC \times \Gamma$. We need to show that $(S, P)$ is a $\Gamma$-contraction, i.e., 
		\begin{equation*}
			\|g(S, P)\| \leq \|g\|_{\infty,\; \Gamma}=\sup\{|g(z_2, z_3)| \ : \ (z_2, z_3) \in \Gamma \}
		\end{equation*}
		for every holomorphic polynomial $g$ in two variables. For a polynomial $g \in \C[z_1,z_2]$, let us set $f(z_1, z_2, z_3)=g(z_2, z_3)$. Using the fact that $\overline{\mathbb{P}} \subseteq \overline{\mathbb{D}} \times \Gamma$, we have
		\begin{equation*}
			\begin{split}
				\|g(S, P)\|=\|f(A, S, P)\|
				& \leq \sup\{|f(z_1, z_2, z_3)| \ : \ (z_1, z_2, z_3) \in \overline{\mathbb{P}} \}\\
				& \leq \sup\{|f(z_1, z_2, z_3)| \ : \ z_1 \in \overline{\mathbb{D}}, \ (z_2, z_3) \in \Gamma \}\\
				&=\sup\{|g(z_2, z_3)| \ : \ (z_2, z_3) \in \Gamma \}.\\
			\end{split}
		\end{equation*}
		Therefore, $(S, P)$ is a $\Gamma$-contraction.
			\end{proof}
	
	It is not difficult to see that if $(a, s, p) \in \PC$ and $\alpha \in \DC$, then $(\alpha a, \alpha s, \alpha^2 p) \in \PC$ and thus \textit{$\PC$ is $(1,1,2)$-quasi-balanced}. See Section 6 in \cite{Agler} for a detailed proof of this. The next proposition generalizes this result to $\mathbb{P}$-contractions.
	
	\begin{prop}\label{prop2.5}
		Let $(A, S, P)$ be a $\mathbb{P}$-contraction on a Hilbert space $\mathcal{H}$. Then $(\alpha A, \alpha S, \alpha^2 P)$ is a $\Pe$-contraction for every $\alpha \in \DC$.
	\end{prop} 
	
	\begin{proof}
		We have that $\PC$ is $(1, 1, 2)$-quasi-balanced. So, for any $\alpha \in \DC$, the map $f_\alpha : \PC \to \PC$ defined by $ f_\alpha(a, s, p)=(\alpha a, \alpha s, \alpha^2 p)$ is analytic. For any holomorphic polynomial $g$ in $3$-variables, we have that 
		\begin{equation*}
			\begin{split}
				\|g(\alpha A, \alpha S, \alpha^2 P)\|& =\|g\circ f_{\alpha}(A, S, P)\| \\
				& \leq \|g\circ f_\alpha\|_{\infty, \PC}\\
				& =\sup\left\{|g(\alpha a, \alpha s, \alpha^2 p)| \ : \ (a, s, p) \in \PC \right\}\\
				& \leq \|g\|_{\infty, \PC} \ .	
			\end{split}
		\end{equation*} 
		Therefore, $(\alpha A, \alpha S, \alpha^2 P)$ is a $\Pe$-contraction. 
	\end{proof}
	
	It is evident from Theorem \ref{thm2.1} that $(\alpha a, s, p) \in \PC$ for every $(a, s, p) \in \PC$ and $\alpha \in \DC$. Thus, by an application of a similar idea as in the previous proposition, we arrive at the following result. 

	\begin{prop}
		Let $(A, S, P)$ be a $\mathbb{P}$-contraction and $\alpha \in \DC$. Then $(\alpha A, S, P)$ is a $\Pe$-contraction.
	\end{prop}

	\begin{proof}
	For any $\alpha \in \DC$, the map $	\phi_\alpha : \PC \to \PC$ defined by $\phi_\alpha(a, s, p)=(\alpha a, s, p)$ is analytic. For any polynomial $q(z_1,z_2,z_3)$, we have
	\begin{equation*}
		\begin{split}
			\|q(\alpha A, S, P)\| =\|q\circ \phi_{\alpha}(A, S, P)\| 
			 \leq \|q\circ \phi_\alpha\|_{\infty, \PC}
			 =\sup\left\{|q(\alpha a, s, p)| \ : \ (a, s, p) \in \PC \right\}
			 \leq \|q\|_{\infty, \PC} \ .	
		\end{split}
	\end{equation*} 
	Consequently, we have that $(\alpha A, S, P)$ is a $\Pe$-contraction. 
\end{proof}
	
	The following observation provides a way to construct $\Pe$-contraction from a given $\Gamma$-contraction.
	
	\begin{prop}
		$(S, P)$ is a $\Gamma$-contraction if and only if $(0, S, P)$ is a $\mathbb{P}$-contraction.
	\end{prop}
	
	\begin{proof}
		Assume that $(S, P)$ is a $\Gamma$-contraction. Let $f$ be a holomorphic polynomial in $3$-variables and let $g(z_2, z_3)=f(0, z_2, z_3)$. Then 
		\begin{equation*}
			\begin{split}
				\|f(0, S, P)\| & =\|g(S, P)\|\\
				& \leq \sup\{|g(z_2, z_3)| \ : \ (z_2, z_3) \in \Gamma  \}  \qquad \text{[since $\Gamma$ is a spectral set for $(S, P)$]}\\ 
				& =\sup \{|f(z_1, z_2, z_3)| \ : \ z_1=0, \ (z_2, z_3) \in \Gamma  \} \\
				& \leq \sup \{|f(z_1, z_2, z_3)| \ : \  (z_1, z_2, z_3) \in \PC  \}. \\
			\end{split}
		\end{equation*}	
		The last inequality follows from the fact that for any $(s, p) \in \Gamma$, the point $(0, s, p) \in \PC$ and this is a consequence of Theorem \ref{thm2.1}. Thus, $(0, S, P)$ is a $\Pe$-contraction. The converse part follows from Proposition \ref{lem2.3}.
	\end{proof}

	\begin{prop}\label{prop2.8}
		A pair $(T, T')$ of operators acting on a Hilbert space $\mathcal{H}$ is a commuting pair of contractions if and only if $(T, 0, T')$ is a $\Pe$-contraction on $\mathcal{H}$.
	\end{prop}
	
	\begin{proof}
		We have shown in the proof of Proposition \ref{lem2.3} that $\PC \subseteq  \DC \times  \Gamma$. Since $\Gamma \subseteq 2\DC \times \DC$, we have that $	\PC \subseteq  \DC \times 2\DC \times \DC$. Let $f_1(z_1, z_2, z_3)=z_1$ and $f_3(z_1, z_2, z_3)=z_3$. So, if $(T, 0, T')$ is a $\Pe$-contraction, then for $j=1,3$ we have  
		\begin{equation*}
			\begin{split} 
				\|f_j(T, 0, T')\| & \leq \sup\{|f_j(z_1, z_2, z_3)| \ : \ (z_1, z_2, z_3) \in \PC \}\\
				&  \leq \sup\{|f_j(z_1, z_2, z_3)| \ : \ (z_1, z_2, z_3) \in \DC \times 2 \DC \times \DC \}\\
				& \leq 1.
			\end{split} 
		\end{equation*}
		Therefore, $\|T\|, \|T'\| \leq 1$. Conversely, let us assume that $(T, T')$ is a commuting pair of contractions. Then it follows from Ando's inequality \cite{Nagy} that 
		\begin{equation}\label{Ando}
			\|p(T, T')\| \leq \|p\|_{\infty, \DC^2},
		\end{equation}
		for every holomorphic polynomial $p$ in $2$-variables. An application of Theorem \ref{thm2.1} gives $\DC \times \{0\} \times \DC \subseteq \PC$.	Let $f$ be a holomorphic polynomial in $3$-variables and let $g(z, w)=f(z, 0, w)$. Then
		\begin{equation*}
			\begin{split}
				\|f(T, 0, T')\| & =\|g(T, T')\|\\
				& \leq \sup\{|g(z_1, z_3)| \ : \ z_1, z_3 \in \DC  \}  \quad \text{[by (\ref{Ando})]}\\ 
				& =\sup \{|f(z_1, z_2, z_3)| \ : \ (z_1, z_2, z_3) \in \DC \times \{0\} \times \DC  \} \\
				& \leq \sup \{|f(z_1, z_2, z_3)| \ : \  (z_1, z_2, z_3) \in \PC  \}. \\
			\end{split}
		\end{equation*}
		Consequently, it follows that $(T, 0, T')$ is a $\mathbb{P}$-contraction.
	\end{proof}

We now show interplay between the pentablock and the Euclidean unit ball $\B_2$ in $\C^2$.
	
	\begin{lem}\label{lem2.10}
		If $(a, s, p) \in \PC$, then $\left(a, s\slash 2 \right) \in \BC$.
	\end{lem}

\begin{proof}	
	Let $(a, s, p) \in \PC$. Then it follows from Theorem \ref{thm2.1} that there is a $2 \times 2$ matrix $A_0=[a_{ij}]$ such that $\|A_0\| \leq 1$ and	$a_{21}=a, \quad a_{11}+a_{22}=s \quad \text{and} \quad a_{11}a_{22}-a_{12}a_{21}=p$.	Let us assume that $|a_{11}| \leq |a_{22}|$. Also, let $e_1=(1, 0)$ and $e_2=(0, 1)$ in $\C^2$. Then 
	\begin{equation*}
		\begin{split}
			|a|^2+\frac{|s|^2}{4} & =|a_{21}|^2+\frac{1}{4}|a_{11}+a_{22}|^2 \\
			&=|a_{21}|^2+\frac{1}{4}\left(|a_{11}|^2+|a_{22}|^2+2Re \ (\overline{a}_{11}a_{22})\right)\\
			&\leq |a_{21}|^2+\frac{1}{4}\left(|a_{11}|^2+|a_{22}|^2+ 2|a_{11}a_{22}|\right)\\	
			&\leq |a_{21}|^2+|a_{22}|^2\\
			&=\|A^*e_2 \|^2 \\
			& \leq 1.\\
		\end{split}
	\end{equation*}
	Similarly, one can prove that $
	|a|^2 + \frac{1}{4}|s|^2  \leq 1$ when $|a_{22}| \leq |a_{11}|$. The proof is complete.

	\end{proof}
	
As expected, this result has an operator theoretic extension, which is given below.

	\begin{prop}\label{prop2.11}
		If $(A, S, P)$ is a $\Pe$-contraction, then $\left(A, \ S \slash 2\right)$ is a $\B_2$-contraction.
	\end{prop}
	
	\begin{proof}
		By Lemma \ref{lem2.10}, the map $f: \PC \to \BC$ defined by $f(a, s, p)=(a, s\slash 2)$ is well-defined and analytic on $\PC$. Now for any $g\in \C[z_1,z_2]$, we have 
		\begin{equation*}
			\begin{split}
				\|g(A, S\slash 2)\|& =\|g\circ f(A, S, P)\| \\
				& \leq \|g\circ f\|_{\infty, \PC}\\
				& =\sup\left\{|g(a, s \slash 2)| \ : \ (a, s, p) \in \PC \right\}\\
				& \leq \sup\left\{|g(a, s \slash 2)| \ : \ (a, s\slash 2) \in \BC \right\}\\
				& = \|g\|_{\infty, \BC} \ .	
			\end{split}
		\end{equation*} 
		Therefore, $\BC$ is a spectral set for $(A, S \slash 2)$.
	\end{proof}

\noindent Putting together everything we obtain the following theorem, which is a main result of this section.

	\begin{thm}\label{lem2.12}
		If $(A, S, P)$ is a $\Pe$-contraction, then
		\begin{enumerate}
			\item[(a)] $(A, S\slash 2)$ is a $\B_2$-contraction;
			\item[(b)] $(S, P)$ is a $\Gamma$-contraction;
			\item[(c)] $(A, P)$ is a commuting pair of contractions.
		\end{enumerate}
	\end{thm}

The converse of Theorem \ref{lem2.12} is not true. Indeed, below we show that there exist $a, s$ and $p$ in $\DC$ such that $(a, s\slash 2) \in \BC , (s, p) \in \Gamma$ but $(a, s, p) \notin \PC$. 
	
	\begin{eg}
		Let $\lambda_1=1$ and $\lambda_2=0$ in $\DC$. Then it follows from the definition of $\Gamma$ that $(s, p)=(\lambda_1+\lambda_2, \lambda_1\lambda_2)=(1, 0) \in \Gamma$. Let $(a, s, p)=\left(\sqrt{3} \slash 2, 1, 0 \right)$. Then obviously	$|a|^2+\frac{1}{4}|s|^2=1$. Thus, $(a, s \slash 2) \in \BC$. Let if possible, $(a, s, p) \in \PC$. Then, by Theorem \ref{thm2.1}, we must have  
		\[
		\frac{\sqrt{3}}{2}=|a| \leq \frac{1}{2}|1-\overline{\lambda}_2\lambda_1|+\frac{1}{2}(1-|\lambda_1|^2)^{\frac{1}{2}}(1-|\lambda_2|^2)^{\frac{1}{2}}= \frac{1}{2},
		\]
		which is a contradiction. Thus, $(a, s, p) \notin \PC$. \qed 
	\end{eg}

	One naturally asks if there is $p \in \DC$ such that $(a, s\slash 2) \in \BC$ implies that $(a, s, p) \in \PC$. Indeed, existence of such a $p$ is guaranteed by the next lemma.

	\begin{lem} \label{basiclem:03}
		$(a, s\slash 2) \in \BC$ if and only if there exists $p \in \mathbb{T}$ such that $(a, s, p) \in \PC$.
	\end{lem}
	
	\begin{proof}
		Let $(a, s\slash 2) \in \BC$. Then $|a| \leq 1$ and $|s| \leq 2$. If $s=0$, then it follows from Theorem \ref{thm2.1} that $(a, 0, p) \in \PC$ for any $p \in \DC$. If $s \ne 0$, we choose $p=s\slash \overline{s}$. Then $(s, p) \in b\Gamma$ by Theorem \ref{Gamma_uni}, as $|p|=1, s=\overline{s}p$ and $|s| \leq 2$. Furthermore, $(a, s \slash 2) \in \BC$, we have that 
		\begin{equation*}
			\begin{split}
				|a| \leq \bigg|1-\frac{\frac{1}{4}|s|^2}{1+\sqrt{1-\frac{1}{4}|s|^2}} \bigg|=\sqrt{1-\frac{1}{4}|s|^2}.
			\end{split}
		\end{equation*}
		Hence, by Theorem \ref{thm2.1}, $(a, s, p) \in \PC$. The converse part follows from Lemma \ref{lem2.10}.
	\end{proof}

Lemma \ref{basiclem:03} can be extended to the class of $\Pe$-contractions consisting of normal operators. Such a $\Pe$-contraction is called a \textit{normal $\Pe$-contraction}. To obtain the desired conclusion, we use the polar decomposition of normal operators. The \textit{polar decomposition theorem} (Theorem 12.35 in \cite{Rudin}) states that if $N$ is a normal operator on a Hilbert space $\mathcal{H}$, then there exists a unitary operator $U$ on $\mathcal{H}$ such that $U$ commutes with any operator that commutes with $N$ and $N=U(N^*N)^{1 \slash 2}=(N^*N)^{1 \slash 2}U$. Set $p(\lambda)=|\lambda|, \ u(\lambda)=\frac{\lambda}{|\lambda|}$ if $\lambda \ne 0$ and $u(0)=1$. Then $p$ and $u$ are bounded Borel functions on $\sigma(N)$. Define $P=p(N)$ and $U=u(N)$. Since $u\overline{u}=1, UU^*=U^*U=I$. Since $\lambda=u(\lambda)p(\lambda)$, the conclusion $N=UP$ follows from the Borel functional calculus. Furthermore, $P$ is a positive operator that satisfies $\|Px\|=\|Nx\|$ for every $x \in \mathcal{H}$ and so, $P=(N^*N)^{1\slash 2}$. If $T \in \mathcal{B}(\mathcal{H})$ commutes with both $N$ and $N^*$, then $T$ commutes with $u(N)=U$.

	\begin{lem}
		Let $(A, S\slash 2)$ be a $\B_2$-contraction consisting of normal operators acting on a space $\mathcal{H}$. Then there exists a unitary $P$ on $\mathcal{H}$ such that $(S, P)$ is a $\Gamma$-unitary and $(A, S, P)$ is a normal $\Pe$-contraction. 
	\end{lem}
	
	\begin{proof}
		It follows from given hypothesis that $S$ is normal and $\|S\| \leq 2$.	Now, from the above discussion we have that there is a unitary $U$ on $\mathcal{H}$ such that $	S=(S^*S)^{1\slash 2}U=U(S^*S)^{1\slash 2}$ and $U$ commutes with any operator that commutes with $S$. Hence, $U$ commutes with $A$. Set $P=U^2$. Then $(A, S, P)$ is a triple of commuting normal operators. Moreover	$
		S^*P=(S^*S)^{1\slash 2}U^*U^2=(S^*S)^{1\slash 2}U=S$ and $\|S\| \leq 2$. Thus, $(S, P)$ is a $\Gamma$-contraction with $P$ being unitary and so, $(S, P)$ is a $\Gamma$-unitary by the virtue of Theorem \ref{Gamma_uni}. Let $(a, s, p)$ in $\sigma_T(A, S, P)$ and let $f(z_1, z_2, z_3)=(z_1, z_2\slash 2)$. Then the spectral mapping theorem yields that 
		\[
		(a, s\slash 2)=f(a, s, p) \in f(\sigma_T(A, S, P))=\sigma_T(f(A, S, P))=\sigma_T(A, S\slash 2).
		\]
		By Proposition \ref{prop2.11}, $(a, s\slash 2) \in \BC$. From the projection property of joint spectrum, it follows that $p \in \sigma(P)$ and so, $|p|=1$ since $P$ is a unitary. Furthermore, for $\beta=\frac{s}{2}$, we have  
		\[
		\bigg|1-\frac{\frac{1}{2}s\overline{\beta}}{1+\sqrt{1-|\beta|^2}} \bigg|=\bigg|1-\frac{\frac{1}{4}|s|^2}{1+\sqrt{1-\frac{1}{4}|s|^2}} \bigg|=\sqrt{1-\frac{1}{4}|s|^2} \geq |a|,
		\]
		where the last inequality follows from the fact that $(a, s\slash 2) \in \BC$. By Theorem \ref{thm2.1}, $(a, s, p) \in \PC$. Consequently, $(A, S, P)$ is a commuting triple of normal operators so that $\sigma_T(A, S, P) \subseteq \PC$.  Hence, $(A, S, P)$ is a normal $\Pe$-contraction by Proposition \ref{basicprop:02}.
	\end{proof}

 We have seen that if $(A, S, 0)$ is a $\Pe$-contraction, then $(A, S \slash 2)$ is a $\B_2$-contraction but the converse does not hold. In this sense, every $\Pe$-contraction gives rise to a $\B_2$-contraction. It is interesting to explore if one can obtain a $\Pe$-contraction from a $\B_2$-contraction. The subsequent results provide an answer to this question. 
	
	\begin{lem}\label{lem2.13}
		Let $(a, s) \in \BC$. Then $(a, s, 0) \in \PC$. 
	\end{lem}
	
	\begin{proof}
		The proof follows from Theorem \ref{thm2.1}. If $(a, s) \in \BC$, then $a, s \in \DC$ and so, $(s, 0)=(\lambda_1+\lambda_2, \lambda_1\lambda_2)$ for $\lambda_1=s$ and $\lambda_2=0$. Thus, $(s, 0) \in \Gamma$. By Theorem \ref{thm2.1}, $(a, s, 0) \in \PC$ if and only if the following inequality holds.
		\[
		|a| \leq  \frac{1}{2}|1-\overline{\lambda}_2\lambda_1|+\frac{1}{2}(1-|\lambda_1|^2)^{\frac{1}{2}}(1-|\lambda_2|^2)^{\frac{1}{2}}=\frac{1}{2}\bigg(1+\sqrt{1-|s|^2}\bigg).
		\]
		Since $|a| \leq \sqrt{1-|s|^2}$, the above inequality holds and so, $(a, s, 0) \in \PC$.	
	\end{proof}

	\begin{prop}\label{prop2.14}
		Let $(A, S)$ be a $\B_2$-contraction. Then $(A, S, 0)$ is a $\Pe$-contraction.
	\end{prop}

	\begin{proof}
		Let $f \in \C[z_1,z_2,z_3]$ be arbitrary polynomial and let $g(z_1, z_2)=f(z_1, z_2, 0)$. Then 
		\begin{equation*}
			\begin{split}
				\|f(A, S, 0)\| & =\|g(A,S)\|\\
				& \leq \sup\{|g(z_1, z_2)| \ : \ (z_1, z_2) \in \BC  \}  \quad \text{[$\because \BC$ is a spectral set for $(A, S)$]}\\ 
				& =\sup \{|f(z_1, z_2, z_3)| \ :  \ (z_1, z_2, z_3) \in \BC \times \{0\} \} \\
				& \leq \sup \{|f(z_1, z_2, z_3)| \ : \  (z_1, z_2, z_3) \in \PC  \}, \\
			\end{split}
		\end{equation*}	
		where the last inequality follows from Lemma \ref{lem2.13}. Consequently,  $(A, S, 0)$ is a $\Pe$-contraction.
	\end{proof}
	
	The converse to Proposition \ref{prop2.14} is not even true for scalars. For example, take $(a, s)=(1\slash 2, 1)$. Then we have that $|a|^2+|s|^2>1$. Since $(s, 0) \in \Gamma$, Theorem \ref{thm2.1} yields that $(a, s, 0) \in \PC$. Thus there exist scalars $a$ and $s$ so that $(a, s, 0) \in \PC$ but $(a, s) \notin \BC$. 

\vspace{0.2cm}

	\section{Operator theory on the unit ball in $\C^n$}\label{Polyball}
	
	\vspace{0.2cm}
	
	\noindent In the previous Section, we have witnessed some interplay between the closed unit ball $\BC$ and $\PC$. Indeed, if $(A, S, P) $ is a $\Pe$-contraction, then $(A, S\slash 2)$ is a $\B_2$-contraction. Also, if $(A, S)$ is a $\B_2$-contraction, then $(A, S, 0)$ is a $\Pe$-contraction. It motivates us to study the operator theoretic aspects of the unit ball $\B^n$ in $\C^n$, where
	\[
	\B_n=\left\{(z_1, \dotsc, z_n) \in \C^n \ : \ |z_1|^2+\dotsc+|z_n|^2<1 \right\}.
	\]
	We shall discuss some useful operator theory on $\overline{\B}_n$. A commuting tuple of operators $(T_1, \dots , T_n)$ for which $\BNC$ is a spectral set is called a $\B_n$-contraction. Contractions have special classes like unitary, isometry, completely non-unitary contraction etc. As we have mentioned in the `Introduction' that a contraction is an operator for which $\overline{\D}$ is a spectral set. A unitary is a normal operator having its spectrum on the unit circle $\T$ and an isometry is the restriction of a unitary to an invariant subspace. Analogously, we have defined $\Pe$-unitary and $\Pe$-isometry associated with the pentablock in Definition \ref{defn:P-unitary}. Similarly, one can define $\B_n$-unitary and $\B_n$-isometry for the Euclidean unit ball $\mathbb B_n$. Note that $\overline{\B}_n$ is a convex compact set and hence is polynomially convex. An interesting fact about $\B_n$ is that its topological boundary $\partial \B_n$ and distinguished boundary $b\B_n$ coincide unlike the polydisc $\D^n$. Needless to mention that 
	\[
	\partial\B_n=\left\{(z_1, \dotsc, z_n) \in \C^n \ : \ |z_1|^2+\dotsc+|z_n|^2 =1 \right\}.
	\]
	The fact that $\partial \B_n=b\B_n$ is explained in \cite{Mackey} (see Example 4.10 in \cite{Mackey} and the discussion thereafter). 
	
	\begin{lem}
		For any $n\geq 1$, $b \B_n=\partial \B_n$.
	\end{lem}
	
 The works of Arveson, Eschmeier and Athavale \cite{ArvesonIII, EschmeierII, AthavaleII} show that the {spherical contractions} naturally occur in the study of operators associated with the unit ball. Before proceeding further, we recall the definition of this class along with its special subclasses from the literature. 
	\begin{defn}
		A commuting tuple $\UT=(T_1, \dotsc, T_n)$ of operators acting on a Hilbert space $\mathcal{H}$ is said to be
		\begin{enumerate}
			\item a \textit{spherical contraction} if $T_1^*T_1 + \dotsc + T_n^*T_n \leq I_\mathcal{H}$;
			
			\item a \textit{spherical unitary} if each $T_j$ is normal and $T_1^*T_1 + \dotsc + T_n^*T_n = I_\mathcal{H}$;
			
			\item a \textit{spherical isometry} if $T_1^*T_1 + \dotsc + T_n^*T_n =I_\mathcal{H}$;
			
			\item a \textit{row contraction} if $(T_1^*, \dotsc, T_n^*)$ is a spherical contraction.
		\end{enumerate}
	\end{defn} 
	Not every spherical contraction or row contraction is a $\B_2$-contraction. We recall Arveson's example from \cite{ArvesonIII} in this context.
	
\begin{eg}
		Consider the pair of multiplication operator $(M_{z_1}, M_{z_2})$ on the \textit{Drury-Arveson space} $H^2_2$, where $H_2^2$ is the reproducing kernel Hilbert space with the kernel 
		\[
		k(z,w)=\frac{1}{1-\langle z, w\rangle} \quad (z, w \in \B_2).
		\] 	
		It follows from Corollary 2 in \cite{ArvesonIII} that $(M_{z_1}, M_{z_2})$ satisfies $M_{z_1}M_{z_1}^*+M_{z_2}M_{z_2}^* \leq I$ but does not form a $\B_2$-contraction as explained in \cite{ArvesonIII}. Thus, $(M_{z_1}^*, M_{z_2}^*)$ is a spherical contraction which is not a $\B_2$-contraction which is same as saying that the row contraction $(M_{z_1},M_{z_2})$ is not a $\B_2$-contraction.\qed 
	\end{eg}

	 However, we shall see below that $\B_n$-contractions and spherical contractions agree at the level of unitaries and isometries. The next result appeared in a discussion in \cite{EschmeierII} whose proof follows directly from the spectral theorem. 
	 
	\begin{thm}\label{prop3.2}
		Let $\underline{U}=(U_1, \dotsc, U_n)$ be a tuple of commuting operators acting on a Hilbert space $\mathcal{H}$. Then $\underline{U}$ is a $\B_n$-unitary if and only if $\underline U$ is a spherical unitary.
	\end{thm}
	
	Recall that a \textit{subnormal tuple} is a tuple $(T_1, \dotsc, T_n)$ of commuting operators that admits a simultaneous normal extension. The following result due to Athavale \cite{AthavaleIII}, which was later proved independently by Arveson \cite{ArvesonIII}, will be useful in the context of this paper.
	
	\begin{lem}[\cite{ArvesonIII}, Corollary 1]\label{subnormal}
		Let $T_1, \dotsc, T_n$ be a set of commuting operators on a Hilbert space $\mathcal{H}$ such that $T_1^*T_1+\dotsc+T_n^*T_n=I_\mathcal{H}$. Then $(T_1, \dotsc, T_n)$ is a subnormal tuple. 
	\end{lem}
The following result shows that a $\mathbb B_n$-isometry is nothing but a spherical isometry and vice-versa.
	
	\begin{thm}[{\cite{AthavaleIII}, Proposition 2}]\label{prop3.5}
		Let $\underline{V}=(V_1, \dotsc, V_n)$ be a tuple of commuting operators acting {on a Hilbert space $\mathcal{H}$}. Then $\underline{V}$ is a $\B_n$ -isometry if and only if  $\underline{V}$ is a spherical isometry.
	\end{thm}
	
One of the most important results in one variable operator theory is the canonical decomposition of a contraction ({see \cite{Langer, NagyFoias} \& CH-I of \cite{Nagy}}), which states that for every contraction $T$ on a Hilbert space $\HS$, the space $\HS$ admits a unique orthogonal decomposition $\HS=\HS_1 \oplus \HS_2$ into joint reducing subspaces of $T$ such that $T|_{\HS_1}$ is a unitary and $T|_{\HS_2}$ is a completely non-unitary contraction. Below we see that a $\B_n$-contraction admits an analogous orthogonal decomposition into a $\B_n$-unitary and a completely non-unitary $\B_n$-contraction. A proof of this result could be found in \cite{EschmeierII}, which was based on Glicksberg-K\"{o}nig-Seever decomposition of a measure. However, we present a short and elementary proof here. 
	
	\begin{thm}[\textbf{Canonical decomposition}, \cite{EschmeierII}] \label{thm:decomp-Ball}
		Let $\underline{T}=(T_1, \dotsc, T_n)$ be a $\B_n$-contraction on a Hilbert space $\mathcal{H}$. Then there is an orthogonal decomposition of $\mathcal{H}$ into joint reducing subspaces $\mathcal{H}_u$ and $\mathcal{H}_c$ of $\underline T$ such that 
		\begin{enumerate}
			\item[(a)] $(T_1|_{\mathcal{H}_u}, \dotsc, T_n|_{\mathcal{H}_u})$ is a $\B_n$-unitary;
			\item[(b)] $(T_1|_{\mathcal{H}_c}, \dotsc, T_n|_{\mathcal{H}_c})$ is a completely non-unitary $\B_n$-contraction.  
		\end{enumerate}
	\end{thm}
	
	\begin{proof}
		For the ease of computations, we shall follow the standard conventions: for a tuple of commuting operators $\underline{T}=(T_1, \dotsc, T_n)$ on space $\mathcal{H}$ and for $\alpha=(\alpha_1, \dotsc, \alpha_n) \in \mathbb{N}^n$, we write $\underline{T}^\alpha=T_1^{\alpha_1} \dotsc T_n^{\alpha_n}$ and $ \underline{T}^{*\alpha}=T_1^{*\alpha_1} \dotsc T_n^{*\alpha_n}$. Also, by \textit{normal tuple}, we mean a tuple of commuting normal operators. It follows from Corollary 4.2 in \cite{Eschmeier} that the closed linear subspace given by
		\[
		\mathcal{H}_0=\bigcap_{\alpha \in \mathbb{N}^n}\bigcap_{ \beta \in \mathbb{N}^n}Ker\bigg[\underline{T}^\alpha\underline{T}^{*\beta}-\underline{T}^{*\beta}\underline{T}^{\alpha}\bigg].
		\]
		is a \textit{maximal} joint reducing subspace for $\underline{T}$ on which $\underline{T}$ acts as a normal tuple. By \textit{maximality}, we mean that if there is a joint reducing subspace $\mathcal{L}$ of $\underline{T}$ such that each $T_j|_{\mathcal{L}}$ is normal, then $\mathcal{L} \subseteq \mathcal{H}_0$. Define 
		\[ 
				\mathcal{H}_u:=\{h \in \mathcal{H}_0 \ : \  \|T_1h\|^2+\dotsc+\|T_nh\|^2=\|h\|^2  \}.
			\]
		Then, $\mathcal{H}_u=\left\{h \in \mathcal{H}_0 \ : \  (I-T_1^*T_1\dotsc-T_n^*T_n)h=0  \right\}$. We show that $\mathcal{H}_u$ is a joint reducing subspace of $\underline{T}$. Let $h \in \mathcal{H}_u$ and let $1 \leq j \leq n$. Since $\mathcal{H}_0$ is a joint reducing subspace, we have that $T_jh \in \mathcal{H}_0$ and $T_j^*h \in \mathcal{H}_0$.	Since $(T_1|_{\mathcal{H}_0}, \dotsc, T_n|_{\mathcal{H}_0})$ is a commuting normal tuple and $(I-T_1^*T_1\dotsc-T_n^*T_n)h=0$, we have
$
		(I-T_1^*T_1\dotsc-T_n^*T_n)T_jh=T_j(I-T_1^*T_1\dotsc-T_n^*T_n)h=0$. We used the fact that if $N_1, N_2$ are commuting normal operators, then $N_1^*N_2=N_2^*N_1$. Therefore, $T_jh \in \mathcal{H}_u$ and similarly $T_j^*h \in \mathcal{H}_u$. Thus, $\mathcal{H}_u$ is a joint reducing subspace of $\underline{T}$ and $(T_1|_{\mathcal{H}_u}, \dotsc, T_n|_{\mathcal{H}_u})$ is a normal tuple satisfying $\|T_1h\|^2+\dotsc \|T_nh\|^2=\|h\|^2 \quad \text{for all} \ h \in \mathcal{H}_u$.	It follows from Theorem \ref{prop3.2} that $(T_1|_{\mathcal{H}_u}, \dotsc, T_n|_{\mathcal{H}_u})$ is a $\mathbb B_n$-unitary on $\mathcal{H}_u$. Setting	$
		\mathcal{H}_c=\mathcal{H}\ominus \mathcal{H}_u,
		$
		we see that $\HS_c$ is a joint reducing subspace of $\underline{T}$. Let $\mathcal{L} \subseteq \mathcal{H}_c$ be a joint reducing subspace of $\underline{T}$ such that $(T_1|_\mathcal{L}, \dotsc, T_n|_\mathcal{L})$ is a $\mathbb B_n$-unitary. Thus $(T_1|_\mathcal{L}, \dotsc, T_n|_\mathcal{L})$ is a normal tuple and the maximality of $\mathcal{H}_0$ implies that $\mathcal{L} \subseteq \mathcal{H}_0$. By Theorem \ref{prop3.2}, we have that $
		\|T_1h\|^2+\dotsc \|T_nh\|^2=\|h\|^2 \quad \text{for all} \ h \in \mathcal{L}$.	Hence, $\mathcal{L} \subseteq \mathcal{H}_u$. Putting everything together, we have that $\mathcal{L} \subset \mathcal{H}_u \cap \mathcal{H}_c=\{0\}$ and so, $\mathcal{L}=\{0\}$. Thus, $(T_1|_\mathcal{L}, \dotsc, T_n|_\mathcal{L})$ is a completely non-unitary $\B_n$-contraction and the proof is complete. 
	\end{proof}
	
		\section{The pentablock unitaries}\label{P-uni}
	
	\vspace{0.2cm}
	
		\noindent Recall that a $\mathbb{P}$-unitary is a normal $\mathbb{P}$-contraction whose joint spectrum lies in the distinguished boundary $b\mathbb{P}$ of the pentablock. In this Section, we find several characterizations for the $\Pe$-unitaries and find their interplay with $\B_2$-unitaries and $\Gamma$-unitaries. First we collect from the existing literature \cite{Agler, Jindal}, similar various characterizations for the points in the distinguished boundary of the pentablock $b\mathbb{P}$. 
	
		\begin{thm} \label{thm:distP}
		For $(a, s, p) \in \mathbb{C}^3$, the following are equivalent:
		\begin{enumerate}
			\item $(a, s, p) \in b\mathbb{P}$;
			\item $(s, p) \in b\Gamma, |a|=\sqrt{1-\frac{1}{4}|s|^2}$ ;
			
			\item There is a unitary matrix $U=[u_{ij}]_{2 \times 2}$ such that 
			$
			u_{11}=u_{22}$ \& $(a, s, p)=(u_{21}, tr(U), det(U)).
			$
		\end{enumerate}
	\end{thm}

	\noindent In other words, we have the following description for the points in the distinguished boundary $b \Pe$:
\[
b\mathbb{P}=\bigg\{(a, s, p) \in \mathbb{C}^3 \ : \ |a|=\sqrt{1-\frac{1}{4}|s|^2}, \ (s, p) \in b\Gamma \bigg\}\,.
			\]
			Interestingly, each of the above characterizations for a point in $b\Pe$ gives rise to a characterization of a $\Pe$-unitary. Also, we have other characterizations in terms of $\B_2$-unitaries and $\Gamma$-unitaries as shown below. We mention here that the equivalence of parts $(1), (2), (5)$ of the next theorem were established independently in \cite{JindalII}.

\begin{thm}\label{P_unitary}
Let $\underline{N}=(N_1, N_2, N_3)$ be a commuting triple of bounded linear operators. Then the following are equivalent:
		\begin{enumerate}
			\item $\underline{N}$ is a $\mathbb{P}$-unitary ;
			\item $N_1^*$ is subnormal, $(N_2, N_3)$ is a $\Gamma$-unitary and $N_1^*N_1=I-\frac{1}{4}N_2^*N_2$ ;
				
			\item $(N_2, N_3)$ is a $\Gamma$-unitary and 
			$
			N_1^*N_1=I-\frac{1}{4}N_2^*N_2$ and $N_1N_1^*=I-\frac{1}{4}N_2N_2^*$ ;

			\item $(N_1, N_2\slash 2)$ is a $\B_2$-unitary and $(N_2, N_3)$ is a $\Gamma$-unitary ;
			
			\item There is a $2 \times 2$ unitary block matrix $U=[U_{ij}]$, where $U_{ij}$ are commuting normal operators, such that $U_{11}=U_{22}$ and 
			$
			\underline{N}=(U_{21}, U_{11}+U_{22}, U_{11}U_{22}-U_{12}U_{21}).
			$
		\end{enumerate}
	\end{thm}

	\begin{proof}
		$(1) \implies (2):$ By definition $N_1, N_2, N_3$ are commuting normal operators and $\sigma_T(N_1,N_2,N_3) \subseteq b\mathbb{P}$. By the spectral mapping theorem, $\sigma_T(N_2, N_3)=P_{2,3}\sigma_T(\underline{N})$ where $P_{2,3}$ is the projection onto the second and third coordinates. It follows from Theorem \ref{thm:distP} and the projection property of the joint spectrum that $\sigma_T(N_2, N_3)\subseteq b\Gamma$. Hence, $(N_2, N_3)$ is a $\Gamma$-unitary. Again, the commutative $C^*$-algebra generated by the commuting normal operators $N_1, N_2, N_3$ is isometrically isomorphic to the $C(\sigma_T(\underline{N}))$ via the continuous functional calculus. The continuous functional takes the coordinate function $z_i$ to $N_i$ for $i=1,2,3$. The coordinate functions satisfy $|z_1|^2=1-\frac{1}{4}|z_2|^2$ on $b\mathbb{P}$ and hence on $\sigma_T(\underline{N})$. Thus, $N_1^*N_1=I-\frac{1}{4}N_2^*N_2$. 
		
		\medskip
		
		\noindent		$(2) \implies (1):$ From the hypothesis that $N_1^*N_1=I-\frac{1}{4}N_2^*N_2$ and  Lemma \ref{subnormal}, it follows that $N_1$ is subnormal. Thus, $N_1^*, N_1$ are subnormal operators and consequently $N_1$ is a normal operator. Let $(a, s, p) \in \sigma_T(\underline{N})$. It follows from the projection property of joint spectrum that $(s, p) \in \sigma_T(N_2, N_3)$. Since $(N_2, N_3)$ is a $\Gamma$-unitary, we have that $(s, p)\in b\Gamma$. The function 
		\[
		f(z_1, z_2, z_3)=|z_1|^2-\bigg(1-\frac{|z_2|^2}{4}\bigg),
		\]
		is continuous on $\sigma_T(\underline{N})$. Then it follows from the continuous functional calculus that 
		\[
		f(N_1, N_2, N_3)=N_1^*N_1-\bigg(I-\frac{1}{4}N_2^*N_2\bigg)=0.
		\]
		Now spectral mapping theorem gives
		$
		\{0\}=\sigma_T(f(\underline{N}))=f(\sigma_T(\underline{N})).
		$
		Since $(a, s, p) \in \sigma_T(\underline{N})$, we have that $f(a, s, p)=0$ and the desired conclusion follows. 

		\medskip

\noindent	$(2) \implies (3):$ Since $N_1^*N_1+\frac{1}{4}N_2^*N_2=I$, Lemma \ref{subnormal} yields that $N_1$ is a normal operator. Hence, 
		$
		N_1N_1^*=N_1^*N_1=I-\frac{1}{4}N_2^*N_2.
		$
		
		\medskip

\noindent $(3) \implies (2):$ This is obvious.
		
		\medskip
		
\noindent $(3) \iff (4):$ Follows from Theorem \ref{prop3.2}. 

\smallskip		

\noindent $(2) \implies (5):$ Set
		$
		U=\begin{bmatrix}
			\frac{1}{2}N_2 & -N_1^*N_3\\
			N_1 & \frac{1}{2}N_2
		\end{bmatrix}.
		$
		Since $N_1, N_2, N_3$ are commuting normal operators, $U_{ij}$ are commuting normal operators. We show that $U$ is a unitary matrix. Since $N_3$ is unitary and $N_2^*N_3=N_2$ as $(N_2, N_3)$ is a $\Gamma$-unitary, we have
		\begin{equation*}
			\begin{split}
				UU^*=\begin{bmatrix}
					\frac{1}{2}N_2 & -N_1^*N_3\\ \\
					N_1 & \frac{1}{2}N_2
				\end{bmatrix}
				\begin{bmatrix}
					\frac{1}{2}N_2^* & N_1^*\\ \\
					-N_1N_3^* & \frac{1}{2}N_2^*
				\end{bmatrix}
				&=\begin{bmatrix}
					\frac{1}{4}N_2N_2^*+ N_1^*N_3N_1N_3^* & \frac{1}{2}N_2N_1^*-\frac{1}{2}N_1^*N_3N_2^*\\ \\
					\frac{1}{2}N_1N_2^*-\frac{1}{2}N_2N_1N_3^* & N_1N_1^*+\frac{1}{4}N_2N_2^*
				\end{bmatrix}\\ \\
				&=\begin{bmatrix}
					\frac{1}{4}N_2^*N_2+ N_1^*N_1N_3^*N_3 & \frac{1}{2}N_1^*N_2-\frac{1}{2}N_1^*N_2^*N_3\\ \\
					\frac{1}{2}N_1N_2^*-\frac{1}{2}N_1N_2N_3^* & N_1^*N_1+\frac{1}{4}N_2^*N_2
				\end{bmatrix}\\ \\
				&=\begin{bmatrix}
					\frac{1}{4}N_2^*N_2+ N_1^*N_1 & \frac{1}{2}N_1^*N_2-\frac{1}{2}N_1^*N_2\\ \\
					\frac{1}{2}N_1N_2^*-\frac{1}{2}N_1N_2^* & \frac{1}{4}N_2^*N_2+ N_1^*N_1
				\end{bmatrix}=\begin{bmatrix}
					I & 0 \\
					0 & I\\
				\end{bmatrix}. 
			\end{split}
		\end{equation*}
		Similarly, we can show that $U^*U=I$. Now $U_{21}=N_1$ and $U_{11}+U_{22}=N_2$. It only remains to show that $U_{11}U_{22}-U_{12}U_{21}=N_3$ which we prove using the fact that $N_2=N_2^*N_3$ in the following way.
		\[
		U_{11}U_{22}-U_{12}U_{21}=\frac{1}{4}N_2^2+N_1^*N_1N_3=\frac{1}{4}N_2^*N_2N_3+N_1^*N_1N_3=\bigg(\frac{1}{4}N_2^*N_2+N_1^*N_1\bigg)N_3=N_3.
		\]  
\smallskip

\noindent $(5) \implies (2):$ Let $U$ be a $2 \times 2$ unitary block matrix $[U_{ij}]$ where $U_{ij}$ are commuting normal operators such that $U_{11}=U_{22}$ and 
		$
		\underline{N}=(U_{21}, U_{11}+U_{22}, U_{11}U_{22}-U_{12}U_{21}).
		$
		It is evident that 
		$
		\|N_2\| \leq \|U_{11}\|+\|U_{22}\| \leq 2\|U\|=2.
		$
		The condition  $U^*U=I$ gives the following set of equations.
		\begin{equation}\label{eq1}
			U_{11}^*U_{11}+U_{21}^*U_{21}=I, \quad
			U_{12}^*U_{12}+U_{22}^*U_{22}=I, 
		\end{equation}
		\begin{equation}\label{eq2}
			U_{12}^*U_{11}+U_{22}^*U_{21}=0, \quad U_{11}^*U_{12}+U_{21}^*U_{22}=0.
		\end{equation}
		Again, $UU^*=I$ provides the following equations.
		\begin{equation}\label{eq3}
			U_{11}U_{11}^*+U_{12}U_{12}^*=I, \quad U_{21}U_{21}^*+U_{22}U_{22}^*=I,
		\end{equation}
		\begin{equation}\label{eq4}
			U_{21}U_{11}^*+U_{22}U_{12}^*=0, \quad 
			U_{11}U_{21}^*+U_{12}U_{22}^*=0 .
		\end{equation}
		Using the above equations, we have the following.
		\begin{equation*}
			\begin{split}
				N_2^*N_3&=(U_{11}^*+U_{22}^*)(U_{11}U_{22}-U_{12}U_{21})\\
				&=(U_{11}^*U_{11})U_{22}-U_{11}^*U_{12}U_{21} +(U_{22}^*U_{22})U_{11}-U_{22}^*U_{12}U_{21}\\
				& =(I-U_{21}^*U_{21})U_{22}-U_{11}^*U_{12}U_{21} +(I-U_{12}^*U_{12})U_{11}-U_{22}^*U_{12}U_{21} \quad [\text{ by } {(\ref{eq1})}] \ \\
				&=U_{22}-(U_{22}U_{21}^*)U_{21} -U_{11}^*U_{12}U_{21}+U_{11}-(U_{12}^*U_{11})U_{12}- U_{22}^*U_{12}U_{21}\\
				&=(U_{11}+U_{22})+(U_{11}^*U_{12})U_{21} -U_{11}^*U_{12}U_{21}+(U_{22}^*U_{12})U_{12}- U_{22}^*U_{12}U_{21} \ \quad [\text{ by } {(\ref{eq2})}]\\
				&=N_2.\\
			\end{split}
		\end{equation*}
		We show that $N_3$ is unitary. Since $N_3$ is normal, it suffices to show that $N_3^*N_3=I$.
		\begin{equation*}
			\begin{split}
				N_3^*N_3&=(U_{11}^*U_{22}^*-U_{12}^*U_{21}^*)(U_{11}U_{22}-U_{12}U_{21})\\
				&=U_{11}^*U_{11}U_{22}^*U_{22}-(U_{11}^*U_{12})U_{22}^*U_{21}-(U_{12}^*U_{11})U_{21}^*U_{22}+U_{12}^*U_{21}^*U_{12}U_{21}\\
				& =U_{11}^*U_{11}U_{22}^*U_{22}+(U_{21}^*U_{22})U_{22}^*U_{21}+(U_{22}^*U_{21})U_{21}^*U_{22}+U_{12}^*U_{21}^*U_{12}U_{21} \quad [\text{ by } (\ref{eq2})]\\
				&=(U_{11}^*U_{11}+U_{21}^*U_{21})U_{22}^*U_{22}+(U_{22}^*U_{22}+U_{12}^*U_{12})U_{21}^*U_{21}\\
				& =U_{22}^*U_{22}+U_{21}^*U_{21} \quad [\text{ by } (\ref{eq1})]\\
				& =I. \quad [\text{ by } (\ref{eq3})]
			\end{split}
		\end{equation*}
		Hence, $(N_2, N_3)$ is a commuting pair of normal operators such that $\|N_2\| \leq 2, N_2^*N_3=N_2$ and $N_3$ is a unitary. Therefore, $(N_2,N_3)$ is a $\Gamma$-unitary. Since $U_{11}=U_{22}$, we have 
		\begin{equation*}
			\begin{split}
				I-\frac{1}{4}N_2^*N_2=I-\frac{1}{4}(U_{11}^*+U_{22}^*)(U_{11}+U_{22})
				=I-U_{22}^*U_{22}
				=U_{21}^*U_{21}
				=N_1^*N_1,\\
			\end{split}
		\end{equation*}
		where the second last equality follows from (\ref{eq3}). The proof is now complete.
	\end{proof}
	
Note that the assumption that $N_1^*$ is subnormal in Theorem \ref{P_unitary} cannot be dropped. Also, unlike operators associated with the symmetrized bidisc and tetrablock (see Theorem 2.5 in \cite{Bhattacharyya} and Theorem 5.4 in \cite{Bhattacharyya-01} respectively), it is not always true that a $\Pe$-unitary is a commuting triple $(N_1,N_2,N_3)$ which is a $\Pe$-contraction and $N_3$ is a unitary. The following example explains all these together.  

	\begin{eg} \label{exm:imp}
		Consider the commuting triple of subnormal operators
		$
		\underline{N}=(N_1, N_2, N_3)=(T_z, \ 0, \ -I)$ on $\ell^2(\mathbb{N})$,	where  $T_z$ is the unilateral shift on $\ell^2(\mathbb{N})$. Then
		\begin{enumerate}
			\item[(a)]  $(0, -I)$ is a $\Gamma$-unitary since $\sigma_T(0, -I)=\{(0, -1)\} \subset b\Gamma$ and;
			\item[(b)]  $N_1^*N_1=T_z^*T_z=I$ which is same as  $N_1^*N_1+\frac{1}{4}N_2^*N_2=I$.
		\end{enumerate}
		Hence, $\underline{N}$ is a commuting triple such that $(N_2, N_3)$ is a $\Gamma$-unitary and $N_1^*N_1=I-\frac{1}{4}N_2^*N_2$ but $N_1^*$ is not subnormal.	Thus, $\underline{N}$ is not a $\Pe$-unitary as $N_1$ is not normal. However, it is true that $(N_1,N_2,N_3)$ is a $\Pe$-contraction, in fact is a $\Pe$-isometry (which follows from Theorem \ref{P_isometry}) with $N_3$ being a unitary. \qed 
	\end{eg}
	
The authors of \cite{JindalII} introduced the notion of a quasi $\Pe$-unitary, defined as a $\Pe$-isometry whose last component is a unitary operator. With respect to this terminology, Example \ref{exm:imp} provides a quasi $\Pe$-unitary which is not a $\Pe$-unitary. Furthermore, Theorem \ref{P_unitary} shows that every $\Pe$-unitary is a quasi $\Pe$-unitary. Hence, the class of quasi $\Pe$-unitaries is strictly larger than the class of $\Pe$-unitaries. The following example shows that one cannot drop the hypothesis $(N_1, N_2\slash 2)$ being a $\B_2$-unitary in Theorem \ref{P_unitary}. Indeed, we show that there exists $(a, s, p) \in \PC$ such that $(s, p) \in b\Gamma$ but $(a, s, p) \notin b\Pe$. 
	\begin{eg}
		Let 
		$
		(a, s, p)=(0, 0, 1).
		$
		Then 
		$
		(s, p) \in b\Gamma, (a, s, p) \in \PC$ and $|a|^2+\frac{1}{4}|s|^2 \ne 1$.
Thus $(a, s, p) \notin b\Pe$. Moreover, it shows that 
		$
		b\Pe \ne \{(a, s, p) \in \PC \ : \ |p|=1 \}.
		$	\qed
	\end{eg}

 The next corollary is an immediate consequence of Theorem \ref{Gamma_uni} and Theorem \ref{P_unitary}. 
	
	\begin{cor}\label{corU_1,U_2}
		$(U_1, U_2)$ is a pair of commuting unitaries if and only if $(U_1, 0,  U_2)$ is a $\Pe$-unitary.
	\end{cor}
	
	 One can easily construct a $\mathbb{P}$-unitary from a given $\Gamma$-unitary in the following way.

	\begin{eg}\label{eg3.5}
		Let $(N_2, N_3)$ be a $\Gamma$-unitary on a Hilbert space $\mathcal{H}$. It follows from the definition of $\Gamma$-contraction that $\frac{1}{2}N_2$ is a contraction. Therefore, we can consider its defect operator which is 
		$
		D_{N_2\slash 2}=(I-\frac{1}{4}N_2^*N_2)^{1\slash 2}.
		$ 
		Since $(N_2, N_3)$ are commuting normal operators, it immediately follows that $D_{N_2\slash 2}$ commutes with $N_1$ and $N_2$. Therefore, $(D_{N_2\slash 2}, N_2, N_3)$ is a triple of commuting normal operators such that $(N_2, N_3)$ is a $\Gamma$-unitary and 
		$
		D_{N_2\slash 2}^2+\frac{1}{4}N_2^*N_2=I.
		$
		Thus, it follows from Theorem \ref{P_unitary} that $(D_{N_2\slash 2}, N_2, N_3)$ is a $\Gamma$-unitary on $\mathcal{H}$.\qed
	\end{eg}
	
We shall use the polar decomposition for normal operators to show that the aforementioned example serves as a prototype of a $\mathbb{P}$-unitary. The proof of the next theorem follows from the polar decomposition theorem,  Theorem \ref{P_unitary} and Example \ref{eg3.5}.

	\begin{thm}
		A commuting triple of operators $\underline{N}=(N_1, N_2, N_3)$ acting on a Hilbert space $\mathcal{H}$ is a $\mathbb{P}$-unitary if and only $(N_2, N_3)$ is a $\Gamma$-unitary and there is a unitary $U$ on $\mathcal{H}$ such that $U$ commutes with $N_2, N_3$ and 
		$
		N_1=UD_{N_2\slash 2}=D_{N_2\slash 2}U.
		$
	\end{thm}

	\begin{proof}
		Let $(N_2, N_3)$ be a $\Gamma$-unitary and let $U$ be a unitary on $\mathcal{H}$ that commutes with $N_2$ and $N_3$. Then $U$ commutes with $N_2^*$ due to Fuglede's theorem and so, 
		\[
		U\left(I-\frac{1}{4}N_2^*N_2\right)=\left(I-\frac{1}{4}N_2^*N_2\right)U.
		\]
		\noindent Consequently, the continuous functional calculus for normal operators yields that $U$ commutes with $D_{N_2\slash 2}$. If we take $N_1=UD_{N_2\slash 2}$, then $N_1$ commutes with $N_2$ and $N_3$ since $U$ and $D_{N_2\slash 2}$ commute with $N_2, N_3$. Also, we have 
		\[
		N_1^*N_1+\frac{1}{4}N_2^*N_2=U^*UD_{N_2\slash2}^2+\frac{1}{4}N_2^*N_2=I.
		\]
		Thus, it follows from Theorem \ref{P_unitary} that $(N_1, N_2, N_3)$ is a $\Pe$-unitary. To see the converse, let $(N_1, N_2, N_3)$ be a $\Pe$-unitary.  By Theorem \ref{P_unitary}, $(N_2, N_3)$ is a $\Gamma$-unitary. Moreover, 
		$
		N_1^*N_1=I-\frac{1}{4}N_2^*N_2=D_{N_2\slash2}^2$ and thus $ (N_1^*N_1)^{1\slash 2}=D_{N_2\slash 2}$.
		It follows from polar decomposition theorem (see the discussion after Lemma \ref{basiclem:03}) that there is a unitary $U$ on $\mathcal{H}$ which commutes with $N_2, N_3$ and 	
		$
		N_1=U(N_1^*N_1)^{1\slash 2}=(N_1^*N_1)^{1\slash 2}U.
		$
		Consequently, $N_1=UD_{N_2\slash 2}=D_{N_2\slash 2}U$ and the proof is complete.
	\end{proof}

 We conclude this section with the following sufficient condition for a $\Pe$-unitary. Recall that a commuting tuple of operators $(T_1, \dotsc, T_n)$ is said to be \textit{doubly commuting} if $T_iT_j^*=T_j^*T_i$ for all $ i \ne j $. 

\begin{prop}
	Let $(A, S, P)$ be a doubly commuting $\Pe$-contraction on $\C^2$ such that 
	$
	\sigma_T(A, S, P)\subseteq b\Pe
	$
	and let $(A, S\slash 2)$ be a spherical contraction. Then $(A, S, P)$ is a $\Pe$-unitary.   
\end{prop}

\begin{proof}
	With respect to a fixed orthonormal basis, we can write $(A, S, P)$ in the following way:
	\[
	A=\begin{bmatrix}
		a_{11} & a_{12}\\
		0 & a_{22}
	\end{bmatrix}, \quad 	
S=\begin{bmatrix}
	s_{11} & s_{12}\\
	0 & s_{22}
\end{bmatrix} \quad  \text{and} \quad 	
P=\begin{bmatrix}
p_{11} & p_{12}\\
0 & p_{22}
\end{bmatrix}.
	\]
Note that 
$
\sigma_T(A, S, P)=\left\{(a_{11}, s_{11}, p_{11}), (a_{22}, s_{22}, p_{22})\right\} \subset b\Pe.
$
By Theorem \ref{P_unitary}, we have 
\begin{equation}\label{doublycommuting}
|p_{11}|=|p_{22}|=1 \quad \text{and} \quad |a_{11}|^2+\frac{1}{4}|s_{11}|^2=1=|a_{22}|^2+\frac{1}{4}|s_{22}|^2.
\end{equation}
It follows from {Theorem} \ref{lem2.12} that $P$ is a contraction and thus, we have
\[
0 \leq I-P^*P=\begin{bmatrix}
	1-|p_{11}|^2 & -\overline{p_{11}}p_{12} \\
	-p_{11}\overline{p_{12}} & 1-|p_{22}|^2-|p_{12}|^2
\end{bmatrix}=\begin{bmatrix}
0 & -\overline{p_{11}}p_{12} \\
-p_{11}\overline{p_{12}} & -|p_{12}|^2
\end{bmatrix}.
\]
So, we have $p_{12}=0$ and thus $P$ is a normal operator such that $P^*P=I$. Therefore, $P$ is unitary and $(S, P)$ is a $\Gamma$-contraction which yields that $(S, P)$ is a $\Gamma$-unitary. It follows from Theorem \ref{thm2.1} that $S-S^*P=0$. A straight forward calculation gives the following: 
\[
S-S^*P=\begin{bmatrix}
	s_{11}-\overline{s_{11}}p_{11} & s_{12}\\
	\overline{s_{12}}p_{11} & s_{22}-\overline{s_{22}}p_{22}
\end{bmatrix}.
\] 	
Hence, $S-S^*P=0$ gives that $s_{12}=0$. Putting everything together, we have that 
\[
	A=\begin{bmatrix}
	a_{11} & a_{12}\\
	0 & a_{22}
\end{bmatrix}, \quad 	
S=\begin{bmatrix}
	s_{11} & 0\\
	0 & s_{22}
\end{bmatrix} \quad  \text{and} \quad 	
P=\begin{bmatrix}
	p_{11} & 0\\
	0 & p_{22}
\end{bmatrix} \quad (|p_{11}|=|p_{22}|=1).
\]	
If $a_{12}=0$, then $(A, S, P)$ is a normal $\Pe$-contraction with $\sigma_T(A, S, P) \subset b\Pe$ and hence, $(A, S, P)$ is a $\Pe$-unitary. Let us assume that $a_{12} \ne 0$. Now, we use the fact that $A$ doubly commutes with $S$ and $P$. A routine computation yields that 
\[
AS^*-S^*A=\begin{bmatrix}
	0 & a_{12}(\overline{s_{22}}-\overline{s_{11}})\\
	0 & 0
\end{bmatrix} \quad \text{and} \quad
 AP^*-P^*A=\begin{bmatrix}
0 & a_{12}(\overline{p_{22}}-\overline{p_{11}})\\
0 & 0
\end{bmatrix}.
\] 	
Thus, $p_{11}=p_{22}$ and $s_{11}=s_{22}$. Now, using the hypothesis that $I-A^*A-\frac{1}{4}S^*S \geq 0$, we have
\[
0 \leq I-A^*A-\frac{1}{4}S^*S=\begin{bmatrix}
	1-|a_{11}|^2-\frac{1}{4}|s_{11}|^2 & -\overline{a_{11}}a_{12} \\
	-a_{11}\overline{a_{12}} & 1-|a_{22}|^2-\frac{1}{4}|s_{11}|^2-|a_{12}|^2\\
\end{bmatrix}
=\begin{bmatrix}
0 & -\overline{a_{11}}a_{12} \\
-a_{11}\overline{a_{12}} & -|a_{12}|^2\\
\end{bmatrix},
\]
where the last inequality follows from (\ref{doublycommuting}). The positive semi-definiteness of $I-A^*A-\frac{1}{4}S^*S$ implies that $-|a_{12}|^2 \geq 0$ and hence, $a_{12}=0$. This is a contradiction. Hence, $a_{12}=0$ and the proof is complete.
\end{proof}

	\section{The pentablock isometries}\label{P-iso}
	
	\vspace{0.2cm}

	\noindent Recall that a $\Pe$-isometry is the restriction of a $\Pe$-unitary $(A,S,P)$ to a joint invariant subspace of $A,S$ and $P$. Thus, a $\Pe$-isometry is a subnormal triple. Note that a tuple of commuting operators $\underline{T}=(T_1, \dotsc, T_m)$ acting on a Hilbert space $\mathcal{H}$ is said to be \textit{subnormal} if there is a Hilbert space $\mathcal{K} \supseteq \mathcal{H}$ and a commuting tuple of normal operators $\underline{N}=(N_1, \dotsc, N_m)$ of commuting normal operators  in $\mathcal{B}(\mathcal{K})$ such that $\mathcal{H}$ is invariant under $N_1, \dotsc, N_m$ and $N_j|_{\HS}=T_j$ for each $j=1, \dots , m$. The tuple $\underline{N}$ is said to be a \textit{normal extension} of $\underline{T}$. It follows from the theory of subnormal operators (see \cite{Lubin, Athavale}) that every subnormal tuple admits a minimal normal extension to the space 
	\[
	\mathcal{K}=\overline{span}\left\{N_1^{*k_1} \dotsc N_m^{*k_m}h \ : \ h \in \mathcal{H} \ \& \ \ k_1, \dotsc, k_m \in \N \cup \{0\} \right\},
	\]
	and a minimal normal extension is unique up to unitary equivalence. We invoke the following results on subnormal operators to prove our theorems of this Section.	
	
	\begin{lem}[\cite{Bram}, Theorem 8]\label{Bram}
		If $S$ is subnormal and $T$ is normal such that $ST=TS$, then $S$ and $T$ have a commuting normal extension. 
	\end{lem} 	
	
	\begin{lem}[\cite{Abrahamese}, pp. 173]\label{Abrahamese}
		If $S$ and $T$ are commuting subnormal operators, then $S$ and $T$ have a commuting normal extension if $\sigma(T)$ is finitely connected and the spectrum of minimal normal extension of $T$ is contained in the topological boundary of $\sigma(T)$.
	\end{lem} 
	
	\begin{lem}[\cite{Athavale}, Proposition 0]\label{Athavale} Let $(S_1, \dotsc, S_n)$ be a commuting $n$-tuple of contractions acting on the space $\mathcal{H}$. Then the following are equivalent:
		\begin{enumerate}
			\item There is a commuting $n$-tuple $(N_1, \dotsc, N_n)$ of normal operators on the space $\mathcal{K} \supseteq \mathcal{H}$ such that $S_i=N_i|_\mathcal{H}, \ i=1, \dotsc, n$.
			\item For every non-negative integers $k_1, \dotsc, k_n$, we have 
			\[
			\underset{\substack{0 \leq p_i \leq k_i\\ 1 \leq i \leq n}}{\sum}(-1)^{p_1+\dotsc+p_n}\binom{k_1}{p_1}\dotsc \binom{k_n}{p_n}S_1^{*p_1}\dotsc S_n^{*p_n}S_1^{p_1}\dotsc S_n^{p_n} \geq 0. 
			\]
		\end{enumerate}  
	\end{lem} 
	
	\begin{lem}[\cite{Lubin}, Corollary 1]\label{Lubin}
		Let $\underline{S}=(S_1, \dotsc, S_n)$ be a subnormal tuple and let $\underline{N}=(N_1, \dotsc, N_n)$ be the minimal normal extension of $\underline{S}$. Then each $N_i$ is unitarily equivalent to the minimal normal extension of $S_i$.
	\end{lem}

	\begin{lem}[\cite{Lubin}, Corollary 2]\label{Lubin2}
		Let $\underline{S}=(S_1, \dotsc, S_n)$ be a subnormal tuple and let $\underline{N}=(N_1, \dotsc, N_n)$ be the minimal normal extension of $\underline{S}$. Then for any $n$-variable polynomial $p, p(\underline{N})$ is unitarily equivalent to the minimal normal extension of $p(\underline{S})$.
	\end{lem}
	
	Also, we recall from the literature the following theorem, which gives characterizations of $\Gamma$-unitaries, appeared in parts in \cite{AglerVII} and \cite{Bhattacharyya}. We shall use this theorem below.

	\begin{thm}\label{G_unitary}
		Let $(S, P)$ be a pair of commuting operators on a Hilbert space $\mathcal{H}$. Then, the following statements are equivalent:
		
		\begin{enumerate}
			\item $(S, P)$ is a $\Gamma$-unitary;
			\item there exist commuting unitary operators $U_1$ and $U_2$ on $\mathcal{H}$ such that
			\[
			S=U_1+U_2, \quad P=U_1U_2; 
			\]
			\item $P$ is unitary, $S=S^*P$ and $r(S) \leq 2$, where $r(S)$ is the spectral radius of $S$;
			\item $(S, P)$ is a $\Gamma$-contraction and $P$ is unitary;
			\item $P$ is a unitary and $S=U+U^*P$ for unitary $U$ commuting with $P$.
		\end{enumerate}
	\end{thm}

We now present a characterization for a $\Pe$-isometry which is also independently proved in \cite{JindalII} (see Theorem 5.2 in \cite{JindalII}). However, our proof to the part $(2) \implies(1)$ of this theorem is significantly different.
	
		\begin{thm}\label{P_isometry}
		Let $(V_1, V_2, V_3)$ be a commuting triple of operators acting on the Hilbert space $\mathcal{H}$. Then the following are equivalent.	
		\begin{enumerate}
			\item $(V_1, V_2, V_3)$ is a $\mathbb{P}$-isometry;
		
			\vspace{0.1cm}
			
			\item $(V_1, V_2\slash 2)$ is a $\B_2$-isometry and $(V_2, V_3)$ is a $\Gamma$-isometry.
		\end{enumerate}
	\end{thm}

	\begin{proof}
		$(1) \implies (2):$ Let $(V_1,V_2,V_3)$ on $\HS$ be a $\Pe$-isometry. Then there is a pentablock unitary $(U_1, U_2, U_3)$ on a Hilbert space $\mathcal{K} \supseteq \mathcal{H}$ such that $\mathcal{H}$ is a joint invariant subspace and $(V_1, V_2, V_3)=(U_1|_{\mathcal{H}}, U_2|_{\mathcal{H}} ,U_3|_{\mathcal{H}})$. Theorem \ref{P_unitary} yields that $U_1$ is normal and $(U_2, U_3)$ is a $\mathbb{P}$-unitary and as a consequence we have that $V_1$ is subnormal and $(V_2, V_3)$ is a $\Gamma$-isometry. Since each $V_j$ is the restriction of $U_j$ to $\mathcal{H}$, we have that 
		$
		V_j^*V_j=P_\mathcal{H}U_j^*U_j|_\mathcal{H}$  for $j=1, 2, 3.
		$
		Therefore, it follows from Theorem \ref{P_unitary} that
		\[
		I-\frac{1}{4}V_2^*V_2-V_1^*V_1=P_{\mathcal{H}}\bigg(I-\frac{1}{4}U_2^*U_2-U_1^*U_1\bigg)\bigg|_\mathcal{H}=0,
		\] 
		and consequently $(V_1, V_2\slash 2)$ is a $\B_2$-isometry by Theorem \ref{prop3.5}.
		
		\medskip
		
\noindent	$(2) \implies (1):$ We first show that $(V_1, V_2, V_3)$ has a simultaneous normal extension which is same as showing that $(V_1, V_2\slash 2, V_3)$ has a simultaneous normal extension. For the ease of writing, we denote 
		$
		(S_1,S_2,S_3)=(V_1, V_2\slash 2,V_3).
		$
	Since $S_1, S_2, S_3$ are all contractions, Lemma \ref{Athavale} yields that there is a commuting triple $(U_1, U_2, U_3)$ of normal operators on a space $\mathcal{K} \supseteq \HS$ such that $\mathcal{H}$ is a joint invariant subspace of $U_1, U_2, U_3$ and $S_i=U_i|_\mathcal{H}$ for $i=1,2,3$ if and only if the following operators are non-negative.

		\begin{enumerate}
			\item $\Delta_i=\underset{0 \leq p_i \leq k_i}{\sum}(-1)^{p_i}\binom{k_i}{p_i}S_i^{*p_i}S_i^{p_i}$ \quad for $k_i \in \mathbb{N} , i=1,2,3;$

			\item $\Delta_{ij}=		\underset{\substack{0 \leq p_i \leq k_i\\ 0 \leq p_j \leq k_j}}{\sum}(-1)^{p_i+p_j}\binom{k_i}{p_i} \binom{k_j}{p_j}S_i^{*p_i}S_j^{*p_j}S_i^{p_i} S_j^{p_j}$ \quad for $k_i , k_j \in \mathbb{N}$, $i,j=1,2,3$ and $i\ne j;$

			\item $\Delta_{123}=		\underset{\substack{0 \leq p_i \leq k_i\\ 1 \leq i \leq 3}}{\sum}(-1)^{p_1+p_2+p_3}\binom{k_1}{p_1}\binom{k_2}{p_2} \binom{k_3}{p_3}S_1^{*p_1}S_2^{*p_2} S_3^{*p_3}S_1^{p_1}S_2^{p_2} S_3^{p_3}$ \quad for $k_1, k_2, k_3 \in \mathbb{N}$.	
		\end{enumerate} 

		We prove the non-negativity of each of the operators defined above.
		\begin{enumerate}
			\item[(a)] Since $S_2$ and $S_3$ are subnormal operators, it follows again from Lemma \ref{Athavale} that $\Delta_2, \Delta_3 \geq 0$. Note that 
			$
			S_1^*S_1+S_2^*S_2=V_1^*V_1+\frac{1}{4}V_2^*V_2=I.
			$ 
			By Lemma \ref{subnormal}, we have that $S_1$ is subnormal and so $\Delta_1 \geq 0$.
			
			\vspace{0.1cm}
			
			\item[(b)] 	Again by Lemma \ref{subnormal}, we have that $(S_1, S_2)$  is a subnormal pair. Thus $\Delta_{12} \geq 0$. Another application of Lemma \ref{Athavale} gives $\Delta_{23} \geq 0$ because, $(V_2, V_3)$ is a $\Gamma$-isometry and hence admits a simultaneous commuting normal extension which is also true for $(S_2, S_3)$.  We now prove that $\Delta_{13} \geq 0$. Since $V_3$ is an isometry, $V_3$ can either be a unitary or has a non-zero shift part.
			
			\smallskip
			
	\noindent \textit{Case I:} Let $S_3=V_3$ be a unitary. Then $(S_1, S_3)$ is a commuting pair of operators such that $S_1$ is subnormal and $S_3$ is normal. By Lemma \ref{Bram}, $S_1$ and $S_3$ have a simultaneous commuting normal extension.
			
			\smallskip
			
			\noindent \textit{Case II:} Suppose that $S_3=V_3$  has a non-zero shift part. In this case $\sigma(S_3)=\overline{\mathbb{D}}$.  The minimal normal extension, say, $N_3$ of $S_3$ is a unitary and hence, we have $\sigma(N_3) \subseteq \mathbb{T}=\partial \mathbb{D}=\partial \sigma(S_3)$. Lemma \ref{Abrahamese} yields that $S_1$ and $S_3$ have a simultaneous commuting normal extension. 
			
			\vspace{0.1cm}

			\noindent In either case, $(S_1, S_3)$ admits a simultaneous commuting normal extension and so, $\Delta_{13} \geq 0$ which follows from Lemma \ref{Athavale}. 
			
			\vspace{0.1cm}
			
			\item[(c)]  It is only remaining to show that $\Delta_{123} \geq 0$. In the computation of $\Delta_{123}$ we use the fact that $S_3^{*p}S_3^p=I$ for every $p \geq 0$, which happens because $V_3$ is an isometry. So, we have the following.
			\begin{equation*}
				\begin{split}
					\Delta_{123}&=		\underset{\substack{0 \leq p_i \leq k_i\\ 1 \leq i \leq 3}}{\sum}(-1)^{p_1+p_2+p_3}\binom{k_1}{p_1}\binom{k_2}{p_2} \binom{k_3}{p_3}S_1^{*p_1}S_2^{*p_2} S_1^{p_1}S_2^{p_2} \\
					&=\bigg[\overset{k_3}{\underset{p_3=0}{\sum}}(-1)^{p_3}\binom{k_3}{p_3}\bigg]\; \bigg[\;\overset{k_1}{\underset{p_1=0}{\sum}}\overset{k_2}{\underset{p_2=0}{\sum}}(-1)^{p_1+p_2}\binom{k_1}{p_1} \binom{k_2}{p_2}S_1^{*p_1}S_2^{*p_2}S_1^{p_1}S_2^{p_2}\bigg]\\
					&=\overset{k_1}{\underset{p_1=0}{\sum}}\overset{k_2}{\underset{p_2=0}{\sum}}(-1)^{p_1+p_2}\binom{k_1}{p_1} \binom{k_2}{p_2}S_1^{*p_1}S_2^{*p_2}S_1^{p_1}S_2^{p_2},\\
				\end{split}
			\end{equation*}
			where the last equality again uses the fact that $\overset{k_3}{\underset{p_3=0}{\sum}}(-1)^{p_3}\binom{k_3}{p_3}=1$ since $k_3 \ne 0$. Now, the non-negativity of $\Delta_{12}$ gives $\Delta_{123} \geq 0$.
		\end{enumerate}
		
\noindent Thus, combining everything together, we see that there is a commuting triple $(U_1, U_2, U_3)$ of normal operators on a space $\mathcal{K} \supseteq \HS$  such that $\mathcal{H}$ is a joint invariant subspace of $U_1, U_2, U_3$ and
		$
		(V_1, V_2, V_3)=(U_1|_\mathcal{H}, U_2|_\mathcal{H}, U_3|_\mathcal{H}).
		$ 
		Without loss of generality, we assume that $(U_1, U_2, U_3)$ on $\mathcal{K}$ is the minimal normal extension of the triple $(V_1, V_2, V_3)$ and the space $\mathcal{K}$ is given by 
		\[
		\overline{\mbox{span}}\{U_1^{*j_1}U_2^{*j_2} U_3^{*j_3}h \ | \ j_1, j_2, j_3 \geq 0, \ h \in \mathcal{H}\}.
		\]
		We claim that $(U_1, U_2, U_3)$ on $\mathcal{K}$ is a $\mathbb{P}$-unitary. It follows from Lemma \ref{Lubin} that each $U_i$ on $\mathcal{K}$ is unitarily equivalent to the minimal normal extension of $V_i$ for $i=1,2,3$. We prove that  $(U_2, U_3)$ is a $\Gamma$-unitary. By Theorem \ref{G_unitary}, it suffices to show that $(U_2, U_3)$ is a $\Gamma$-contraction and $U_3$ is a unitary. Note that $U_3$ is a unitary by being the minimal normal extension of the isometry $V_3$. Let $g$ be a holomorphic polynomial in $2$-variables. Let $f(z_1, z_2, z_3)=g(z_2, z_3)$.  It follows from Lemma \ref{Lubin2} that $f(U_1, U_2, U_3)$ is unitarily equivalent to the minimal normal extension, say, $\widetilde{N}$ of $\widetilde{S}=f(V_1, V_2, V_3)$. Bram \cite{Bram} proved that a subnormal operator satisfies the spectral inclusion relation, that is $\sigma(\widetilde{N}) \subseteq \sigma(\widetilde{S})$. Since $\widetilde N$ is normal, we have that 
		\begin{equation*}
			\begin{split}
				\|\widetilde{N}\|=\sup\{|\lambda| \ : \ \lambda \in \sigma(\widetilde{N})\}
				\leq \sup\{|\lambda| \ : \ \lambda \in \sigma(\widetilde{S})\}
				& = \sup\{|\lambda| \ : \ \lambda \in \sigma(f(V_1, V_2, V_3))\}\\
				& = \sup\{|\lambda| \ : \ \lambda \in \sigma(g(V_2, V_3))\}\\
				& = \sup\{|\lambda| \ : \ \lambda \in g(\sigma_T(V_2, V_3))\}\\
				& \leq \sup\{|\lambda| \ : \ \lambda \in g(\Gamma)\}\\
				&=\|g\|_{\infty, \Gamma}.
			\end{split}
		\end{equation*} 
		Since $g(U_2, U_3)=f(U_1, U_2, U_3)$ and $f(U_1,U_2,U_3)$ is unitarily equivalent to $\widetilde{N}$, we must have 
		\[
		\|g(U_2, U_3)\| =\|\widetilde{N}\| \leq \|g\|_{\infty, \Gamma}.
		\]
Hence, $(U_2, U_3)$ is a $\Gamma$-contraction. Thus, $(U_2, U_3)$ is a $\Gamma$-unitary. We now show that $(U_1,U_2)$ is a $\B_2$-unitary, that is
		$
		U_1^*U_1+\frac{1}{4}U_2^*U_2-I=0.
		$
		Let $h \in \mathcal{H}$. Then
		\begin{small}
			\begin{equation*}
				\begin{split}
				& \|(U_1^*U_1+\frac{1}{4}U_2^*U_2-I)h\|^2 = \\
					& \bigg(\|U_1^2h\|^2+\frac{1}{4}\|U_1U_2h\|^2-\|U_1h\|^2\bigg)
					+\bigg(\|U_1U_2h\|^2+\frac{1}{4}\|U_2^2h\|^2-\|U_2h\|^2\bigg)
					+\bigg(\|U_1h\|^2+\frac{1}{4}\|U_2h\|^2-\|h\|^2\bigg)\\	
					&=\bigg(\|V_1^2h\|^2+\frac{1}{4}\|V_1V_2h\|^2-\|V_1h\|^2\bigg)
					+\bigg(\|V_1V_2h\|^2+\frac{1}{4}\|V_2^2h\|^2-\|V_2h\|^2\bigg)
					+\bigg(\|V_1h\|^2+\frac{1}{4}\|V_2h\|^2-\|h\|^2\bigg)\\	 
	&  \rightline{$[\because V_i=U_i|_{\mathcal{H}}]$} \\
	& =0,\\
				\end{split}
			\end{equation*}
		\end{small} 
		where, the last equality holds because 
		\[
		\|V_1h\|^2+\frac{1}{4}\|V_2h\|^2-\|h\|^2=\langle (V_1^*V_1-I+\frac{1}{4}V_2^*V_2)h, h \rangle=0, 
		\]
		for every $h \in \mathcal{H}$. Therefore, 
		$
		(U_1^*U_1+\frac{1}{4}U_2^*U_2-I)h=0
		$
		for every $h \in \mathcal{H}$. From the definition of $\mathcal{K}$, it follows that $U_1^*U_1+\frac{1}{4}U_2^*U_2-I=0$ on $\mathcal{K}$. Theorem \ref{P_unitary} yields that $(U_1, U_2, U_3)$ on $\mathcal{K}$ is a $\mathbb{P}$-unitary. Hence, $(V_1, V_2, V_3)$ is a $\mathbb{P}$-isometry by being the restriction of the $\Pe$-unitary $(U_1, U_2, U_3)$ to that joint invariant subspace $\HS$. The proof is now complete.
	\end{proof}

 The following two results are direct consequences of Theorem \ref{Gamma_uni} and Theorem \ref{P_isometry}. 
	
	\begin{cor}\label{V_1, 0, V_2}
		$(V_1, V_2)$ is a pair of commuting isometries if and only if $(V_1, 0,  V_2)$ is a $\Pe$-isometry.
	\end{cor}
	
	\begin{cor}\label{P_unitaryII}
		Let $\underline{N}=(N_1, N_2, N_3)$ be a commuting triple of bounded linear operators. Then $\underline{N}$ is a $\mathbb{P}$-unitary if and only if both $(N_1, N_2, N_3)$ and $(N_1^*, N_2^*, N_3^*)$ are $\Pe$-isometries.
	\end{cor}

	\begin{proof}
		The necessary condition follows from Theorem \ref{P_unitary} and Theorem \ref{P_isometry}. Let us assume that $(N_1, N_2, N_3)$ and $(N_1^*, N_2^*, N_3^*)$ are $\Pe$-isometries. In particular, each $N_j$ and $N_j^*$ are subnormal operators and so, each $N_j$ is normal. Therefore,  $(N_1, N_2, N_3)$ is a commuting triple of normal operators such that
		$N_1^*N_1+\frac{1}{4}N_2^*N_2=I$ and $N_3^*N_3=N_3N_3^*=I$. Theorem \ref{Gamma_uni} and Theorem \ref{P_unitary} yield that $(N_1, N_2, N_3)$ is a $\mathbb{P}$-unitary.
	\end{proof}

Our next theorem is an analogue of the Wold decomposition for a $\mathbb{P}$-isometry. Indeed, we show that a $\Pe$-contraction orthogonally decomposes into a $\Pe$-unitary and a pure $\Pe$-isometry. Before that we state a result due to Agler and Young from \cite{AglerVII}, which will be useful. 

	\begin{thm}[\cite{AglerVII}, Theorem 2.6] \label{Wold_Gamma_Isometry}
		Let $(S,P)$ be a $\Gamma$-isometry and $\mathcal{H}=\mathcal{H}_1 \oplus \mathcal{H}_2$ be the Wold decomposition of the isometry $P$ into its unitary part $P|_{\mathcal{H}_1}$ and the shift part $P|_{\mathcal{H}_2}$. Then $\mathcal{H}_1$ and $\mathcal{H} _2$ are reducing subspaces for $S$ such that $(S|_{\mathcal{H}_1}, P|_{\mathcal{H}_1})$ is a $\Gamma$-unitary and $(S|_{\mathcal{H}_2}, P|_{\mathcal{H}_2})$ is a pure $\Gamma$-isometry i.e. $(S|_{\mathcal{H}_2}, P|_{\mathcal{H}_2})$ is a $\Gamma$-isometry and $P|_{\mathcal{H}_2}$ is a unilateral shift operator.
	\end{thm}

	\begin{thm}\textbf{(Wold decomposition for a $\mathbb{P}$-isometry).}\label{Wold}
		Let $(V_1, V_2, V_3)$ be a $\mathbb{P}$-isometry on a Hilbert space $\mathcal{H}$. Then, there is a unique orthogonal decomposition $\mathcal{H}=\mathcal{H}_u \oplus \mathcal{H}_c$ such that $\mathcal{H}_u$ and $\mathcal{H}_c$ are reducing subspaces of $V_1, V_2, V_3$ and that $(V_1|_{\HS_u}, V_2|_{\HS_u},V_3|_{\HS_u})$ is a $\Pe$-unitary and $(V_1|_{\HS_c}, V_2|_{\HS_c},V_3|_{\HS_c})$ is a pure $\Pe$-isometry.
	\end{thm}

	\begin{proof}
		Let $V_3=P_3 \oplus Q_3$ with respect to the orthogonal decomposition $\mathcal{H}=\mathcal{H}_1 \oplus \mathcal{H}_2$ be the Wold decomposition of the isometry $V_3$ such that $P_3$ on $\HS_1$ is a unitary and $Q_3$ on $\HS_2$ is a pure isometry i.e. a unilateral shift. It follows from Theorem \ref{P_isometry} that $(V_2, V_3)$ is a $\Gamma$-isometry. Therefore, Theorem \ref{Wold_Gamma_Isometry} yields that $\mathcal{H}_1, \mathcal{H}_2$ are reducing subspaces for $V_2$ such that $(V_2|_{\mathcal{H}_1}, V_3|_{\mathcal{H}_1})$ is a $\Gamma$-unitary and $(V_2|_{\mathcal{H}_2}, V_3|_{\mathcal{H}_2})$ is a pure $\Gamma$-isometry. Thus, if $V_2|_{\HS_1}=P_2$ and $V_2|_{\HS_2}=Q_2$, then with respect to the decomposition $\HS=\HS_1 \oplus \HS_2$,
		\[
		\quad V_2=	\begin{bmatrix}
			P_2 & 0\\
			0 & Q_2\\
		\end{bmatrix}, \quad V_3=\begin{bmatrix}
			P_3 & 0\\
			0 & Q_3\\
		\end{bmatrix}.
		\]
		Suppose
		\[
		V_1=\begin{bmatrix}
			P_{1} & A_{12}\\
			A_{21} & Q_{1}\\
		\end{bmatrix}, \quad \text{with respect to } \HS=\HS_1 \oplus \HS_2.
		\]
		By the commutativity of $V_1$ with $V_3$, we have that $A_{12}Q_3=P_3A_{12}$ and  $A_{21}P_3=Q_3A_{21}$. It is well-known that (see Lemma 2.13 in \cite{Bhattacharyya}) that no non zero operator can have such intertwining relation since $P_3$ is a unitary and $Q_3$ is a shift. Consequently, $A_{12}=A_{21}=0$. Thus, with respect to the decomposition $\mathcal{H}=\mathcal{H}_1 \oplus \mathcal{H}_2$, we have
		\[
		V_1=\begin{bmatrix}
			P_1 & 0\\
			0 & Q_1\\
		\end{bmatrix}.
		\]
		Evidently, $P_{1}$ and ${Q_{1}}$ are contractions. Thus, we have that the commuting triple $(P_1, P_2, P_3)$ on $\mathcal{H}_1$ is a $\mathbb{P}$-isometry such that $(P_2, P_3)$ is a $\Gamma$-unitary and $P_1$ is a subnormal contraction. We further decompose the space $\mathcal{H}_1$. It follows from Lemma 3.1 in \cite{Morrel} that the space
		\[
		\mathcal{H}_u=\overset{\infty}{\underset{j=0}{\bigcap}}\text{Ker}\bigg(P_1^{*j}P_1^j-P_1^jP_1^{*j}\bigg) \subseteq \mathcal{H}_1,
		\]
		is a reducing subspace for $P_1$ on which $P_1$ acts a normal operator. Since $P_2$ and $P_3$ are normal operators that commute with $P_1$, Fuglede's theorem \cite{Fuglede} yields that $P_2^*$ and $P_3^*$ also commute with $P_1$. Consequently, $P_2$ and $P_3$ doubly commute with $(P_1^{*j}P_1^j-P_1^jP_1^{*j})$ for every $j \geq 0$. Thus $\mathcal{H}_u$ is a reducing subspace for $P_2$ and $P_3$ as well. With respect to the orthogonal decomposition $\mathcal{H}_1=\mathcal{H}_u \oplus \mathcal{H}_{0}$, let the block matrix form of each $P_i$ be given by
		\[
		P_1=\begin{bmatrix}
			U_1 & 0\\
			0 & B_1\\
		\end{bmatrix}, \quad P_2=\begin{bmatrix}
			U_2 & 0\\
			0 & B_2\\
		\end{bmatrix} \quad \text{and} \quad P_3=\begin{bmatrix}
			U_3 & 0\\
			0 & B_3\\
		\end{bmatrix}.
		\]
		Since $U_1, U_2, U_3$ are restrictions of $V_1, V_2, V_3$ respectively to the common reducing subspace $\mathcal{H}_u$, therefore, 
		$
		U_1^*U_1=I-\frac{1}{4}U_2^*U_2.
		$
		Hence, it follows from Theorem \ref{P_unitary} that the commuting triple $(U_1, U_2, U_3)$ of normal operators acting on $\mathcal{H}_u$ is indeed a $\mathbb{P}$-unitary. With respect to the decomposition of the whole space $\mathcal{H}=\mathcal{H}_u \oplus \mathcal{H}_0 \oplus \mathcal{H}_2$, the block matrix form of each $V_i$ is given by
		\[
		V_1=\begin{bmatrix}
			U_1 & 0 & 0\\
			0 & B_1 & 0\\
			0 & 0 & Q_1
		\end{bmatrix} , \quad 
		V_2=\begin{bmatrix}
			U_2 & 0 & 0\\
			0 & B_2 & 0\\
			0 & 0 & Q_2
		\end{bmatrix}, \quad
		V_3=\begin{bmatrix}
			U_3 & 0 & 0\\
			0 & B_3 & 0\\
			0 & 0 & Q_3
		\end{bmatrix},
		\]
		which we re-write as 
		\[
		V_1=\begin{bmatrix}
			U_1 & 0 \\
			0 & S_1 
		\end{bmatrix} , \quad 
		V_2=\begin{bmatrix}
			U_2 & 0 \\
			0 & S_2 
		\end{bmatrix}, \quad
		V_3=\begin{bmatrix}
			U_3 & 0 \\
			0 & S_3
		\end{bmatrix},
		\]
		with respect to the decomposition $\mathcal{H}=\mathcal{H}_u\oplus \mathcal{H}_c$, where, $\mathcal{H}_c=\mathcal{H}_0 \oplus \mathcal{H}_2$. Denoting 
		\[
		(V_1|_{\mathcal{H}_c}, V_2|_{\mathcal{H}_c}, V_3|_{\mathcal{H}_c})=(S_1, S_2, S_3),
		\] 
		we now show that $(S_1,S_2,S_3)$ is a pure $\mathbb{P}$-isometry. If possible, let there be a closed joint reducing subspace say, $\mathcal{L}$ of $\mathcal{H}_c$ on which $(S_1, S_2, S_3)$ acts as a $\mathbb{P}$-unitary. Theorem \ref{P_unitary} shows that $(S_2, S_3)$ is a $\Gamma$-unitary and hence $S_3$ is a unitary. Consequently, $\mathcal{L} \subseteq \mathcal{H}_1$. On the subspace $\mathcal{L}$, the contraction $P_1$  acts as a normal operator because, $\mathcal{L} \subseteq \mathcal{H}_1$ implying that $S_1|_{\mathcal{L}}=V_1|_{\mathcal{L}}=P_1|_{\mathcal{L}}$. Since $\mathcal{H}_u$ is the maximal closed subspace of $\mathcal{H}_1$ which reduces $P_1$ and on which $P_1$ acts as a normal, therefore, $\mathcal{L} \subseteq \mathcal{H}_u$. Hence, $\mathcal{L} \subseteq \mathcal{H}_u \cap \mathcal{H}_c=\{0\}$. This also shows that any closed joint reducing subspace of $\mathcal{H}$ on which $(V_1, V_2, V_3)$ acts as a $\mathbb{P}$-unitary must be contained in $\mathcal{H}_u$ and in this sense, $\mathcal{H}_u$ is maximal.
		
\smallskip		
		
		We now prove the uniqueness of the decomposition. Let  $\mathcal{H}=\mathcal{L}_u \oplus \mathcal{L}_c$ be an arbitrary decomposition of $\mathcal{H}$ with the properties in the statement of the theorem. The maximality of $\mathcal{H}_u$ implies that $\mathcal{L}_u \subseteq \mathcal{H}_u$. The spaces $\mathcal{H}_u$ and $\mathcal{L}_u$ reduce each $V_i$, therefore, the same is true for $\mathcal{H}_u \ominus \mathcal{L}_u$ and $(V_1, V_2, V_3)|_{\mathcal{H}_u \ominus \mathcal{L}_u}$ is a $\mathbb{P}$-unitary. Since $\mathcal{H}_u \ominus \mathcal{L}_u \subseteq \mathcal{H} \ominus \mathcal{L}_u=\mathcal{L}_c$ and since $(V_1, V_2, V_3)$ is a pure $\mathbb{P}$-isometry on $\mathcal{L}_c$, we have that $\mathcal{H}_u \ominus \mathcal{L}_u=\{0\}$. This shows that $\mathcal{H}_u =\mathcal{L}_u$ and the desired conclusion follows.
	\end{proof}
	
	Note that it is not necessary that a $\Pe$-isometry is a commuting triple $(A,S,P)$ that is a $\Pe$-contraction with $P$ being an isometry unlike the isometries associated with the symmetrized bidisc and tetrablock, (see Theorem 2.14 in \cite{Bhattacharyya} and Theorem 5.7 in \cite{Bhattacharyya-01} respectively). The following example shows this.
	
\begin{eg} \label{exm:imp-1}
We recall Example \ref{exm:imp} first. It follows from Theorem \ref{P_isometry} that the commuting triple of subnormal operators
		$
		\underline{N}=(N_1, N_2, N_3)=(T_z, \ 0, \ -I)$ on $\ell^2(\mathbb{N})$, where $T_z$ is the unilateral shift on $\ell^2(\mathbb{N})$, is a $\Pe$-isometry but not a $\Pe$-unitary. Thus, its adjoint $(T_z^*,0,-I)$ is a $\Pe$-contraction by Lemma \ref{basiclem:01} whose last component, that is $-I$ is an isometry. However, it follows from Theorem \ref{P_isometry} that $(T_z^*,0,-I)$ is not a $\Pe$-isometry.
		\qed
		\end{eg}
		
A $\Pe$-isometry with its last component being a pure isometry, plays major role in determining the structure of a $\Pe$-isometry. Indeed, in the proof of Theorem \ref{Wold}, the last component of the $\Pe$-isometry $(Q_1,Q_2,Q_3)=(V_1|_{\HS_2}, V_2|_{\HS_2}, V_3|_{\HS_2})$ was a pure isometry. We conclude this Section by producing a concrete operator model for a $\Pe$-isometry whose last component is a pure isometry. For this we need to mention the highly efficient machinery called the fundamental operator of a $\Gamma$-contraction. It was proved in \cite{Bhattacharyya} that to every $\Gamma$-contraction $(S,P)$ there is a unique operator $F \in \mathcal B(\mathcal D_P)$ with numerical radius $\omega(F) \leq 1$ such that
\begin{equation} \label{eqn:funda-01}
	S-S^*P=D_PFD_P, 
	\end{equation}
where $D_P=(I-P^*P)^{\frac{1}{2}}$ and $\mathcal D_P=\overline{Ran}\, D_P$. Indeed, a major role in the operator theory of the symmetrized bidisc is played by this unique operator. For this reason $F$ was named the \textit{fundamental operator} of the $\Gamma$-contraction $(S,P)$.

		\begin{thm}
		Let $(V_1, V_2, V_3)$ be a $\Pe$-contraction on a Hilbert space $\mathcal{H}$. If $V_3$ is a pure isometry, then there is a unitary operator $U: \mathcal{H} \to H^2(\mathcal{D}_{V_3^*})$ and a partial isometry $V$ on $\mathcal{H}$ such that 
		\[
		V_1=VU^*D_{\frac{1}{2}T_\phi}U, \quad V_2=U^*T_\phi U , \quad \text{and} \quad V_3=U^*T_z U, 
		\]
		where $\phi(z)=F_*^*+F_*z$ and $F_*$ is the fundamental operator of $(V_2^*, V_3^*)$. 
	\end{thm}

	\begin{proof}
		It follows from Theorem \ref{P_isometry} that $(V_2, V_3)$ is a $\Gamma$-isometry. Since $V_3$ is a pure isometry, Theorem 2.16 in \cite{PalI} yields that there is a unitary operator $U: \mathcal{H} \to H^2(\mathcal{D}_{V_3^*})$ such that 
		\[
		V_2=U^*T_\phi U \quad \text{and} \quad  V_3=U^*T_z U, \quad \phi(z)=F_*^*+F_*z,
		\]
		$F_*$ being the fundamental operator of $(V_2^*, V_3^*)$. Again by Theorem \ref{P_isometry}, we have $V_1^*V_1=I-\frac{1}{4}V_2^*V_2$ and hence
		$
		V_1^*V_1=U^*\left(I-\frac{1}{4}T_\phi^*T_\phi\right)U.
		$
		It follows by iteration that
		\[
		(V_1^*V_1)^n=U^*\left(I-\frac{1}{4}T_\phi^*T_\phi\right)^nU \quad \text{for} \quad n=0,1 ,2, \dotsc.
		\]
		Consequently, we have that 
		$
		p(V_1^*V_1)=U^*p\left(I-\frac{1}{4}T_\phi^*T_\phi\right)U
		$
		for every polynomial $p(\lambda)=a_0+a_1\lambda+ \dotsc + a_n\lambda^n$. Choose a sequence of polynomials $p_m(\lambda)$ which tends to the function $\lambda^{1\slash 2}$ on the interval $0 \leq \lambda \leq 1$. The sequence of operators $p_m(T)$ converges then converges to $T^{1\slash 2}$ in operator norm. Therefore, 
		$
		(V_1^*V_1)^{1\slash 2}=U^*\left(I-\frac{1}{4}T_\phi^*T_\phi\right)^{1\slash 2}U.
		$
		Recall that for every $T \in \mathcal{B}(\mathcal{H})$, there is a partial isometry $V$ on $\mathcal{H}$ such that $T=V(T^*T)^{1 \slash 2}$. Therefore, there is a partial isometry $V$ on $\mathcal{H}$ such that 
		\[
		V_1=V(V_1^*V_1)^{1\slash 2}=VU^*\left(I-\frac{1}{4}T_\phi^*T_\phi\right)^{1\slash 2}U=VU^*D_{\frac{1}{2}T_\phi}U
		\]
		and this completes the proof. 
		\end{proof}
	
	\section{Canonical decomposition of a $\mathbb{P}$-contraction}\label{decomp}
	
\vspace{0.2cm}

\noindent As we have mentioned in Section \ref{Polyball} (see the discussion before Theorem \ref{thm:decomp-Ball}) that every contraction $T$ acting on a Hilbert space $\HS$ admits a canonical decomposition $T_1\oplus T_2$ with respect to $\HS=\HS_1 \oplus \HS_2$, where $T_1$ is a unitary and $T_2$ is a completely non-unitary contraction. The maximal reducing subspace $\HS_1$ on which $T$ acts as a unitary is given by
	\begin{equation*}
		\begin{split}
			\mathcal{H}_1 =\{h \in \mathcal{H}: \|T^nh\|=\|h\|=\|T^{*n}h\|, \ n=1,2, \dotsc \} = \underset{n \in \mathbb{Z}}{\bigcap} Ker D_{T(n)},\\	
		\end{split}
	\end{equation*}
	where, 
	\[
	D_{T(n)}= \left\{
	\begin{array}{ll}
		(I-T^{*n}T^n)^{1\slash 2} & n \geq 0 \\
		(I-T^{|n|}T^{*|n|})^{1\slash 2} & n <0 \,.\\
	\end{array} 
	\right. 
	\]
	An analogous result holds for a pair of doubly commuting contractions as the following result shows.
		
	\begin{thm}[{\cite{SloncinskiII} \& \cite{Pal-II}, Theorem 4.2}]\label{lem5.2}
		For a pair of doubly commuting contractions $P,Q$ acting on a Hilbert space $\HS$, if $Q=Q_1 \oplus Q_2$ is the canonical decomposition of $Q$ with respect to the orthogonal decomposition $\HS=\HS_1 \oplus \HS_2$, then $\mathcal{H}_1, \mathcal{H}_2$ are reducing subspaces for $P$.
	\end{thm}
	
	In \cite{AglerVII}, Agler and Young  proved an analogue of canonical decomposition for a $\Gamma$-contraction $(S, P)$. Interestingly, such a decomposition of $(S, P)$ corresponds to the canonical decomposition of the contraction $P$ as the following theorem shows.
	
	\begin{thm}[\cite{AglerVII}, Theorem 2.8]\label{thm5.3}
		Let $(S, P)$ be a $\Gamma$-contraction on a Hilbert space $\mathcal{H}$. Let $\mathcal{H}_1$ be the maximal subspace of $\mathcal{H}$ which reduces $P$ and on which $P$ is unitary. Let $\mathcal{H}_2=\mathcal{H} \ominus \mathcal{H}_1$. Then $\mathcal{H}_1, \mathcal{H}_2$ reduces $S, (S|_{\mathcal{H}_1}, P|_{\mathcal{H}_1})$ is a $\Gamma$-unitary and $(S|_{\mathcal{H}_2}, P|_{\mathcal{H}_2})$ is a $\Gamma$-contraction for which $P|_{\mathcal{H}_2}$ is a completely non-unitary contraction.
	\end{thm}
	
	Here we present a canonical decomposition of a $\Pe$-contraction. Indeed, we show that every $\Pe$-contraction admits an orthogonal decomposition into a $\Pe$-unitary and a completely non-unitary $\Pe$-contraction. We divide our proof into two parts. First we prove the result for a normal $\Pe$-contraction.	
	
	\begin{prop}\label{normal}
		Let $(A, S, P)$ be a normal $\mathbb{P}$-contraction on a Hilbert space $\mathcal{H}$. Then there is an orthogonal decomposition $\mathcal{H}=\mathcal{H}_u\oplus \mathcal{H}_c$ into joint reducing subspaces $\mathcal{H}_u, \mathcal{H}_c$ of $A, S, P$ such that $(A|_{\mathcal{H}_u}, S|_{\mathcal{H}_u}, P|_{\mathcal{H}_u})$ is a $\mathbb{P}$-unitary and $(A|_{\mathcal{H}_c}, S|_{\mathcal{H}_c}, P|_{\mathcal{H}_c})$ is a completely non-unitary $\mathbb{P}$-contraction. Moreover, $\mathcal{H}_u$ is the maximal closed joint reducing subspace of $\mathcal{H}$ on which $(A, S, P)$ acts as a $\Pe$-unitary.
	\end{prop}

	\begin{proof}
		It follows from Proposition \ref{lem2.3} that $(S, P)$ on $\HS$ is a $\Gamma$-contraction and thus $P$ and $S\slash 2$ are contractions. Let $\mathcal{H}=\mathcal{H}_1\oplus \mathcal{H}_2$ be the canonical decomposition of $P$. An application of Theorem \ref{lem5.2} yields that $\mathcal{H}_1, \mathcal{H}_2$ are reducing subspaces for $A$ and $S$. Let 
		\[
		A=\begin{bmatrix}
			A_1 & 0 \\
			0 & A_2
		\end{bmatrix}, \quad S=\begin{bmatrix}
			S_1 & 0 \\
			0 & S_2
		\end{bmatrix} \quad \text{and} \quad P=\begin{bmatrix}
			P_1 & 0 \\
			0 & P_2
		\end{bmatrix}
		\]
		with respect to the decomposition $\mathcal{H}=\mathcal{H}_1\oplus \mathcal{H}_2$, so that $P_1$ is unitary and $P_2$ is completely non-unitary. Theorem \ref{thm5.3} yields that $(S_1, P_1)$ is a $\Gamma$-unitary on $\mathcal{H}_1$. We now further decompose $\mathcal{H}_1$ into an orthogonal sum of two joint reducing subspaces, say, $\mathcal{H}_1=\mathcal{H}_{11} \oplus \mathcal{H}_{12}$ so that $(A_1|_{\mathcal{H}_{11}}, S_1|_{\mathcal{H}_{11}}, P_1|_{\mathcal{H}_{11}})$ is a $\mathbb{P}$-unitary. To do this, Theorem \ref{P_unitary} implies that we must have 
		$
		A_1^*A_1=I-\frac{1}{4}S_1^*S_1$ on $\ \mathcal{H}_{11}$.	Indeed, we take 
		\[
		\mathcal{H}_{11}= Ker\bigg[I-A_1^{*}A_1-\frac{1}{4}S_1^*S_1 \bigg] \subseteq \mathcal{H}_1.
		\] 
		Since $A_1, S_1, P_1$ are commuting normal operators, it follows that $A_1, S_1, P_1$ doubly commutes with the operator $(I-A_1^{*}A_1-\frac{1}{4}S_1^*S_1)$. Therefore, $\mathcal{H}_{11}$ reduces $A_1, S_1, P_1$ and hence, $A, S, P$.  For any $x \in \mathcal{H}_{11}$, we have 
		$	(I-A_1^*A_1-\frac{1}{4}S_1^*S_1)x=0.
		$
		Then by Theorem \ref{P_unitary}, $(A_1|_{\mathcal{H}_{11}}, S_1|_{\mathcal{H}_{11}}, P_1|_{\mathcal{H}_{11}})$ is a $\mathbb{P}$-unitary. Setting
		$
		\mathcal{H}_{u}=\mathcal{H}_{11}$ and $\mathcal{H}_c=\mathcal{H}\ominus \mathcal{H}_{11}$, it is immediate that $\mathcal{H}_c$ reduces $A, S, P$. We show that $(A_1|_{\mathcal{H}_c}, S_1|_{\mathcal{H}_c}, P_1|_{\mathcal{H}_c})$ is a completely non-unitary $\mathbb{P}$-contraction. Assume that there is a closed joint reducing subspace, say, $\mathcal{L}$ of $\mathcal{H}$ on which $(A, S, P)$ acts as a $\mathbb{P}$-unitary. Theorem \ref{P_unitary} implies that $(S, P)$ is a $\Gamma$-unitary and hence $P$ is a unitary. Consequently, $\mathcal{L} \subseteq \mathcal{H}_1$. On the subspace $\mathcal{L}$, the triple $(A, S, P)$ acts as $\mathbb{P}$-unitary. Thus, Theorem \ref{P_unitary} yields that 
		$
		A^*A-\frac{1}{4}S^*S-I=0$ on $\ \mathcal{L}$. Consequently, $\mathcal{L} \subseteq \mathcal{H}_u$. Hence, $\mathcal{H}_u$ is the maximal closed joint reducing subspace of $\mathcal{H}$ restricted to which $(A, S, P)$ acts as a $\Pe$-unitary. Let $\mathcal{L}$ be a closed joint reducing subspace of $\mathcal{H}_c$ on which $(A, S, P)$ acts as a $\mathbb{P}$-unitary. Since $\mathcal{H}_u$ is a maximal such subspace,  $\mathcal{L} \subseteq \mathcal{H}_u$. Hence, 
		\[
		\mathcal{L} \subseteq \mathcal{H}_u \cap \mathcal{H}_c=\{0\}
		\]
		and so, $(A|_{\mathcal{H}_c}, S|_{\mathcal{H}_c}, P|_{\mathcal{H}_c})$ is a completely non-unitary $\mathbb{P}$-contraction. The proof is now complete.
	\end{proof}
	
		Now we are going to present the main theorem of this Section, the canonical decomposition of a $\Pe$-contraction. We shall follow the same notations as in Section \ref{Polyball}, that is to say for a commuting operator tuple $\underline{T}=(T_1, \dotsc, T_n)$ and for $\alpha=(\alpha_1, \dotsc, \alpha_n) \in \mathbb{N}^n$, we write 
	$
	\underline{T}^\alpha=T_1^{\alpha_1} \dotsc T_n^{\alpha_n}$ and $\underline{T}^{*\alpha}=T_1^{*\alpha_1} \dotsc T_n^{*\alpha_n}. 
	$
	
	\begin{thm}\textbf{(Canonical decomposition of a $\Pe$-contraction).}
		Let $(A, S, P)$ be a $\mathbb{P}$-contraction on a Hilbert space $\mathcal{H}$. Then $\mathcal{H}$ admits an orthogonal decomposition $\HS=\mathcal{H}_u\oplus \mathcal{H}_c$ into joint reducing subspaces $\mathcal{H}_u, \mathcal{H}_c$ of $A, S, P$ such that $(A|_{\mathcal{H}_u}, S|_{\mathcal{H}_u}, P|_{\mathcal{H}_u})$ is a $\mathbb{P}$-unitary and $(A|_{\mathcal{H}_c}, S|_{\mathcal{H}_c}, P|_{\mathcal{H}_c})$ is a completely non-unitary $\mathbb{P}$-contraction.
	\end{thm}
	
	\begin{proof}
		Let $\underline{T}=(A, S, P)$ and let
		\[
		\mathcal{H}_0=\bigcap_{\alpha \in \mathbb{N}^3}\bigcap_{ \beta \in \mathbb{N}^3}Ker\bigg[\underline{T}^\alpha\underline{T}^{*\beta}-\underline{T}^{*\beta}\underline{T}^{\alpha}\bigg].
		\]
		It follows from Eschmeier's work (see Corollary 4.2 in \cite{Eschmeier}) that $\mathcal{H}_0$ is the largest joint reducing subspace of $A,S,P$ on which $(A, S, P)$ acts as a commuting triple of normal operators. Let 
		$
		(A_0, \ S_0, \ P_0)=(A|_{\mathcal{H}_0}, S|_{\mathcal{H}_0}, P|_{\mathcal{H}_0})
		$
		which is a $\Pe$-contraction consisting of normal operators. Proposition \ref{normal} yields that $\mathcal{H}_0$ admits an orthogonal decomposition 
		$
		\mathcal{H}_0=\mathcal{H}_u\oplus \mathcal{H}_{0c}
		$ 
		such that $\mathcal{H}_u$ and $\mathcal{H}_{0c}$ reduce $A_0, S_0, P_0$ and hence $A, S, P$. Moreover, $(A|_{\mathcal{H}_u}, S|_{\mathcal{H}_u}, P|_{\mathcal{H}_u})$ is a $\mathbb{P}$-unitary and $(A|_{\mathcal{H}_{0c}}, S|_{\mathcal{H}_{0c}}, P|_{\mathcal{H}_{0c}})$ is a completely non-unitary $\mathbb{P}$-contraction. Let 
		$
		\mathcal{H}_c=\mathcal{H}\ominus\mathcal{H}_u.
		$
		Let if possible, there is a non-zero closed joint reducing subspace $\mathcal{L}$ of $\mathcal{H}_c$ on which $(A, S, P)$ acts as a $\Pe$-unitary. In particular, $(A, S, P)$ acts as a commuting triple of normal operators on $\mathcal L$ and thus $\mathcal{L} \subseteq \mathcal{H}_0$. Since $\mathcal{H}_u$ is the maximal joint reducing subspace of $\mathcal{H}_0$ restricted to which $(A, S, P)$ is a $\Pe$-unitary, we have that $\mathcal{L} \subseteq \mathcal{H}_u$. Therefore, $\mathcal{L} \subseteq \mathcal{H}_u \cap \mathcal{H}_c=\{0\}$. Hence, 
		$
		(A|_{\mathcal{H}_c}, S|_{\mathcal{H}_c}, P|_{\mathcal{H}_c})
		$
		is a completely non-unitary $\mathbb{P}$-contraction and proof is complete. \end{proof}

	\section{Dilation of a $\mathbb{P}$-contraction}\label{dilation}
	
	\vspace{0.2cm}
	
	\noindent In this Section, we find a necessary and sufficient condition such that a $\Pe$-contraction $(A,S,P)$ admits a $\Pe$-isometric dilation on the minimal dilation space of the contraction $P$ and then explicitly construct such a dilation. Note that the existence of a $\Pe$-isometric dilation guarantees the existence of a $\Pe$-unitary dilation as every $\Pe$-isometry extends to a $\Pe$-unitary. However, our result does not ensure the success of rational dilation on the pentablock. Again, the pentablock is a polynomially convex domain. So, the Oka-Weil theorem (see CH-7 of \cite{Oka-Weil}) yields that the algebra of polynomials is dense in the rational algebra $\mathcal R(\overline{\Pe})$. Also, rational dilation for a $\Pe$-contraction is just a $\Pe$-unitary dilation of it. Below we define $\Pe$-isometric and $\Pe$-unitary dilations of a $\Pe$-contraction.
	\begin{defn}
		Let $(A, S, P)$ be a $\mathbb{P}$-contraction on a Hilbert space $\mathcal{H}$. A $\Pe$-isometry (or $\Pe$-unitary) $(X, T, V)$ acting on a Hilbert space $\mathcal{K} \supseteq \mathcal{H}$ is said to be a \textit{$\mathbb{P}$-isometric dilation} (or a $\Pe$-\textit{unitary dilation}) of $(A, S, P)$ if 
				\[
		A^{i}S^{j}P^{k}=P_\mathcal{H}X^{i}T^{j}V^{k}|_\mathcal{H}, \quad \text{for all } \ \ i, j, k \in \mathbb{N} \cup \{0\}.
		\] 
		Moreover, such a $\Pe$-isometric dilation is called \textit{minimal} if 
		\[
		\mathcal{K}=\overline{\text{span}}\{X^iT^jV^kh \ : \ h \in \mathcal{H} \ \text{and} \ i, j, k \in \mathbb{N} \cup \{0\} \}.
		\]
		However, the \textit{minimality} of a $\Pe$-unitary dilation demands $i,j,k$ to vary over the set of integers $\Z$.
	\end{defn} 
	
We begin with a few preparatory results associated with $\Pe$-contractions.

\begin{prop}
		If a $\mathbb{P}$-contraction $(A, S, P)$ defined on a Hilbert space $\mathcal{H}$ has a $\mathbb{P}$-isometric dilation, then it has a minimal $\mathbb{P}$-isometric dilation.
	\end{prop}
	
	\begin{proof}
		Let $(X, T, V)$ on $\mathcal{K} \supseteq \mathcal{H}$ be a $\mathbb{P}$-isometric dilation of $(A, S, P)$. Let $\mathcal{K}_0$ be the space defined as 
		\[
		\mathcal{K}_0= \overline{\text{span}}\{X^iT^jV^kh \ : \ h \in \mathcal{H} \ \text{and} \ i, j, k \in \mathbb{N} \cup \{0\} \}.
		\] 
		It is easy to see that $\mathcal{K}_0$ is invariant under $X^{i}, T^j$ and $V^k$, for any non-negative integers $i, j, k$. Therefore, if we denote the restrictions of $X, T, V$ to the common invariant subspace $\mathcal{K}_0$ by $X_1, T_1, V_1$ respectively, we get 
		$
		X_1^iy=X^iy$, $T_1^jy=T^jy$, and $V_1^ky=V^ky$ for all $y \in \mathcal{K}_0$. Hence
		\[
		\mathcal{K}_0= \overline{\text{span}}\{X_1^iT_1^jV_1^kh \ : \ h \in \mathcal{H} \ \text{and} \ i, j, k \in \mathbb{N} \cup \{0\} \}.
		\] 
		Therefore, for any non-negative integers $i, j$ and $k$, we have that
		$
		P_\mathcal{H}(X_1^iT_1^jV_1^k)h=A^iS^jP^kh$, for all $h \in \mathcal{H}$. Since $(X, T, V)$ on $\mathcal{K}$ is a $\mathbb{P}$-isometry, it follows from the definition that there is a larger space $\widetilde{\mathcal{K}}$ containing $\mathcal{K}$ and a $\mathbb{P}$-unitary $(U_1, U_2, U_3)$ on $\widetilde{\mathcal{K}}$ such that $\mathcal{K}$ is a common invariant subspace of $\mathcal{K}$ and  
		$
		(X, T, V)=(U_1|_{\mathcal{K}}, U_2|_{\mathcal{K}}, U_3|_{\mathcal{K}}).
		$
		Since $\mathcal{K}_0$ is a subspace of $\mathcal{K}$ which is invariant under $X, T$ and $V$, we have that 
		\[
		(X_1, T_1, V_1)=(X|_{\mathcal{K}_0}, T|_{\mathcal{K}_0}, V|_{\mathcal{K}_0})=(U_1|_{\mathcal{K}_0}, U_2|_{\mathcal{K}_0}, U_3|_{\mathcal{K}_0}).
		\]
		Therefore, $(X_1, T_1, V_1)$ on $\mathcal{K}_0$ is a minimal $\mathbb{P}$-isometric dilation of $(A, S, P)$.
	\end{proof}

	\begin{prop}
		Let $(X, T, V)$ on $\mathcal{K}$ be a $\mathbb{P}$-isometric dilation of a $\mathbb{P}$-contraction $(A, S, P)$ on $\mathcal{H}$. If $(X, T, V)$ is minimal, then $(X^*, T^*, V^*)$ is a {$\mathbb{P}$-co-isometric extension} of $(A^*, S^*, P^*)$.
	\end{prop}
	
	\begin{proof}
	
		We first prove that $AP_\mathcal{H}=P_\mathcal{H}X, SP_\mathcal{H}=P_\mathcal{H}T$ and $PP_\mathcal{H}=P_\mathcal{H}V$. Clearly 
		\[
		\mathcal{K}=\overline{\text{span}}\{X^{i}T^{j}V^{k}h \ : \ h \in \mathcal{H} \ \text{and} \ i, j, k \in \mathbb{N} \cup \{0\} \}.
		\]
		Now for $h \in \mathcal{H}$, we have
		\[
		AP_\mathcal{H}(X^{i}T^{j}V^{k}h)=A(A^{i}S^{j}P^{k}h)=A^{i+1}S^{j}P^{k}h=P_\mathcal{H}(X^{i+1}T^{j}V^{k}h)=P_\mathcal{H}X(X^{i}T^{j}V^{k}h).
		\]	
		Thus we get that $AP_\mathcal{H}=P_\mathcal{H}X$ and similarly, we can show that $SP_\mathcal{H}=P_\mathcal{H}T$ and $PP_\mathcal{H}=P_\mathcal{H}V$. Also for $h \in \mathcal{H}$ and $k \in \mathcal{K}$, we have
		\[
		\langle A^*h, k \rangle =\langle P_\mathcal{H}A^*h, k \rangle =\langle A^*h, P_\mathcal{H}k \rangle =\langle h, AP_\mathcal{H}k \rangle =\langle h, P_\mathcal{H}Xk \rangle =\langle X^*h, k \rangle. 
		\]
		Hence, $A^*=X^*|_\mathcal{H}$ and similarly $S^*=T^*|_\mathcal{H}$ and $P^*=V^*|_\mathcal{H}$. The proof is complete. 
	\end{proof}
	
We have explained in Section \ref{Prelims} the connection between $\Pe$-contractions and the operator theory on the symmetrized bidisc. Indeed, Proposition \ref{lem2.3} shows that if $(A,S,P)$ is a $\Pe$-contraction then $(S,P)$ is a $\Gamma$-contraction. For this reason, the success of rational dilation on $\Gamma$ (see \cite{AglerII, Bhattacharyya}) will play a major role in the dilation of a $\Pe$-contraction. In \cite{Bhattacharyya}, an explicit $\Gamma$-isometric dilation was constructed for any $\Gamma$-contraction. Below we mention this dilation theorem from \cite{Bhattacharyya}.

	\begin{thm}[\cite{Bhattacharyya}, Theorem 4.3]\label{thm6.2} 
		Let $(S,P)$ be a $\Gamma$-contraction on a Hilbert space $\mathcal{H}$. Let $F$ be the fundamental operator of $(S,P)$, that is, unique solution of the operator equation $S-S^*P=D_PXD_P$ as in $(\ref{eqn:funda-01})$. Consider the operators $T_F, V_0$ defined on $\mathcal{H}\bigoplus \ell^2(\mathcal{D}_P)$ by 
		\begin{equation*}	
			\begin{split}
				& T_F(x_0, x_1, x_2, \dotsc)=(Sh_0, F^*D_Ph_0+Fh_1, F^*h_1+Fh_2, F^*h_2+Fh_1, \dotsc)\\
				& V_0(x_0, x_1, x_2, \dotsc)=(Ph_0, D_Ph_0,h_1, h_2, \dotsc). 
			\end{split}  
		\end{equation*}
		Then
		\begin{enumerate}
			\item $(T_F, V_0)$ is a $\Gamma$-isometric dilation of $(S,P)$.
			\item If $(\widehat{T}, V_0)$ on $\mathcal H \oplus l^2(\mathcal D_P)$ is a $\Gamma$-isometric dilation of $(S, P)$, then $\widehat{T}=T_F$.
			\item If $(T, V)$ is a $\Gamma$-isometric dilation of $(S, P)$ where $V$ is a minimal isometric dilation of $P$, then $(T, V)$ is unitarily equivalent to $(T_F, V_0)$.
		\end{enumerate}
	\end{thm}

It is evident from the definition that with respect to the decomposition $\mathcal H \oplus l^2(\mathcal D_P) = \mathcal H \oplus \mathcal D_P \oplus \mathcal D_P \oplus \dots $, the operators $T_F, V_0$ have the following form:
	\[
	T_F=  \begin{bmatrix} 
		S & 0 & 0 & 0 & \dotsc \\
		F^*D_{P} & F & 0 & 0 & \dotsc \\
		0 & F^* & F & 0 & \dotsc\\
		0 & 0 & F^* & F & \dotsc\\
		\dotsc & \dotsc & \dotsc & \dotsc & \dotsc \\
	\end{bmatrix},  \quad 
	V_0=  \begin{bmatrix} 
		P & 0 & 0 & 0 & \dotsc \\
		D_{P} & 0 & 0 & 0 & \dotsc \\
		0 & I & 0 & 0 & \dotsc\\
		0 & 0 & I & 0 & \dotsc\\
		\dotsc & \dotsc & \dotsc & \dotsc & \dotsc \\
	\end{bmatrix}. 
	\]
	It is evident that if $(X, T, V)$ is a $\Pe$-isometric dilation of a $\Pe$-contraction $(A, S, P)$, then $(T,V)$ is a $\Gamma$-isometric dilation of the $\Gamma$-contraction $(S,P)$. Again, Theorem \ref{thm6.2} tells us that if $V$ is the minimal isometric dilation of $P$, then $(T, V)$ is unitarily equivalent to $(T_F, V_0)$. Taking cue from this, we find a necessary and sufficient condition such that $(A, S, P)$ dilates to a $\Pe$-isometry $(X, T_F, V_0)$ acting on the space $\mathcal{H} \oplus \ell^2(\mathcal{D}_P)$.

	\begin{thm}\label{thm:main-dilation}
		Let $(A, S, P)$ be a $\mathbb{P}$-contraction on $\mathcal{H}$. Then $(A, S, P)$ admits a $\mathbb{P}$-isometric dilation $(X,T,V)$ with $V$ being the minimal isometric dilation of $P$ if and only if there exist sequences $(X_{n1})_{n=2}^\infty$ and $(X_n)_{n=2}^\infty$ of operators acting on $\mathcal{H}$ and $\mathcal{D}_P$ respectively such that the following hold:
		\begin{enumerate}
			\item $X_{n1}=X_{n+1, 1}P+X_{n+1}D_P \quad $ for $n=2,3, \dotsc$ ,\\
			\item $X_{21}P+X_2D_P=D_PA$,\\
			\item $X_{21}S+X_2F^*D_P=F^*D_PA+FX_{21}$,\\
			\item $X_{n1}S+X_nF^*D_P=F^*X_{n-1, 1}+FX_{n1} \quad $ for $n=3,4,\dotsc$, \\
			\item $X_2F=FX_2$,\\
			\item $X_nF+X_{n-1}F^*=F^*X_{n-1}+FX_n \quad $ for $n=3,4,\dotsc$,\\
			\item $I-A^*A-\frac{1}{4}S^*S=\overset{\infty}{\underset{n=2}{\sum}}X_{n1}^*X_{n1}+\frac{1}{4}D_PFF^*D_P$,\\
			\item $\overset{\infty}{\underset{n=2}{\sum}}X_{n}^*X_{n}=I-\frac{1}{4}(F^*F+FF^*)$,\\
			\item $\overset{\infty}{\underset{n=2}{\sum}}X_{n}^*X_{n+k,1}=0=\overset{\infty}{\underset{n=2}{\sum}}X_{n+k+1}^*X_{n} \quad $ for $k=1,2, \dotsc$,\\
			\item 		$\overset{\infty}{\underset{n=2}{\sum}}X_{n1}^*X_{n}+\frac{1}{4}D_PF^2=0=\overset{\infty}{\underset{n=2}{\sum}}X_{n+1}^*X_{n}+\frac{1}{4}F^2$.
		\end{enumerate}
	\end{thm}

	\begin{proof}
		Suppose that a $\mathbb{P}$-contraction $(A, S, P)$ acting on a Hilbert space $\mathcal{H}$ dilates to a $\mathbb{P}$-isometry $(X, T, V)$ on a Hilbert space $\mathcal{K}$ with $V$ being the minimal isometric dilation of $P$. Since the minimal isometric dilation of a contraction is unique up to a unitary, without loss of generality let us assume that $V$ is the Sch\"{a}ffer's minimal isometric dilation of $P$, that is
		\[  
V=  \begin{bmatrix} 
			P & 0 & 0 & 0 & \dotsc \\
			D_{P} & 0 & 0 & 0 & \dotsc \\
			0 & I & 0 & 0 & \dotsc\\
			0 & 0 & I & 0 & \dotsc\\
			\dotsc & \dotsc & \dotsc & \dotsc & \dotsc \\
		\end{bmatrix}	
	\]	
	acting on the space $\mathcal K =\HS \oplus \mathcal D_P \oplus \mathcal D_P \oplus \dots$ . Then $V=\begin{bmatrix}
			P & 0\\
			C_3 & E_3
		\end{bmatrix}$ with respect to the decomposition $\mathcal{H} \oplus \ell^2(\mathcal{D}_P)$ of $\mathcal{K}$, where, 
		\[
		C_3=\begin{bmatrix}
			D_P\\
			0 \\
			0\\
			\dotsc
		\end{bmatrix} \ : \ \mathcal{H} \to \mathcal{D}_P \oplus \mathcal{D}_P \oplus \mathcal{D}_P \oplus \dotsc \quad \& \quad E_3= \begin{bmatrix} 
			0 & 0 & 0 & \dotsc \\
			I & 0 & 0 &  \dotsc\\
			0 & I & 0 &  \dotsc\\
			\dotsc & \dotsc  & \dotsc & \dotsc \\
		\end{bmatrix} \ \ \text{on} \ \mathcal{D}_P\oplus \mathcal{D}_P \oplus \mathcal{D}_P \oplus \dotsc .
		\]
		Using the $2 \times 2$ block matrix form of $V$ and the fact that $X$ and $T$ commute with $V$, it follows from straightforward computation that with respect to the decomposition $\mathcal{K}=\mathcal{H} \oplus \ell^2(\mathcal{D}_P)$, the operators $X$ and $T$ have the following operator matrix forms: 
		\[
		X=\begin{bmatrix}
			A & 0\\
			C_1 & E_1
		\end{bmatrix} \quad \& \quad T=\begin{bmatrix}
			S & 0\\
			C_2 & E_2
		\end{bmatrix},
		\] 
		for some $C_i$ and $E_i$ and $1\leq i \leq 2$. It then follows from Theorem \ref{thm6.2} that 
		\[
		T=T_F=  \begin{bmatrix} 
			S & 0 & 0 & 0 & \dotsc \\
			F^*D_{P} & F & 0 & 0 & \dotsc \\
			0 & F^* & F & 0 & \dotsc\\
			0 & 0 & F^* & F & \dotsc\\
			\dotsc & \dotsc & \dotsc & \dotsc & \dotsc \\
		\end{bmatrix}  \quad  \text{and} \quad
		V=V_0=  \begin{bmatrix} 
			P & 0 & 0 & 0 & \dotsc \\
			D_{P} & 0 & 0 & 0 & \dotsc \\
			0 & I & 0 & 0 & \dotsc\\
			0 & 0 & I & 0 & \dotsc\\
			\dotsc & \dotsc & \dotsc & \dotsc & \dotsc \\
		\end{bmatrix}, 
		\]
		with respect to the decomposition $\mathcal{K} = \mathcal{H} \oplus \mathcal{D}_P\oplus \mathcal{D}_P\oplus \dotsc$. Suppose that with respect to the same decomposition of $\mathcal{K}, X$ has the operator matrix form given by 
		\[
		X=  \begin{bmatrix} 
			A & 0 & 0 & 0 & \dotsc \\
			X_{21} & X_{22} & X_{23} & X_{24} & \dotsc \\
			X_{31} & X_{32} & X_{33} & X_{34} & \dotsc \\
			X_{41} & X_{42} & X_{43} & X_{44} & \dotsc \\
			\dotsc & \dotsc & \dotsc & \dotsc & \dotsc \\
		\end{bmatrix}.
		\]
		Some routine but laborious calculations yield the following: 
		
		\[
		XV_0=  \begin{bmatrix} 
			AP & 0 & 0 & 0 & \dotsc \\
			X_{21}P+X_{22}D_P & X_{23} & X_{24} & X_{25} &  \dotsc \\
			X_{31}P+X_{32}D_P & X_{33} & X_{34} & X_{35} &  \dotsc\\
			X_{41}P+X_{42}D_P & X_{43} & X_{44} & X_{45} & \dotsc\\
			\dotsc & \dotsc & \dotsc & \dotsc & \dotsc \\
		\end{bmatrix}
		\quad \text{and} \quad 
		V_0X=  \begin{bmatrix} 
			PA & 0 & 0 & 0 &  \dotsc \\
			D_{P}A & 0 & 0 & 0 & \dotsc \\
			X_{21} & X_{22} & X_{23} & X_{24} & \dotsc \\
			X_{31} & X_{32} & X_{33} & X_{34} & \dotsc \\
			\dotsc & \dotsc & \dotsc & \dotsc & \dotsc \\
		\end{bmatrix}.
		\]	
		This shows that $X$ and $V_0$ commute if and only if 
		\begin{enumerate}
			\item[(a)] $X_{ij}=0 \ $ for all $2 \leq i<j$,
			\item[(b)] $X_{ij}=X_{i+k, j+k} \ $ for all $i,j \geq 2$ and $k \in \mathbb{N}$,
			\item[(c)] $X_{21}P+X_{22}D_P=D_PA$,
			\item[(d)] $X_{n1}=X_{n+1, 1}P+X_{n+1, 2}D_P$ for all $n \geq 2$.   
		\end{enumerate}
		Hence, the operator matrix form of $X$ with respect to the decomposition of $\mathcal{K}=\mathcal{H} \oplus \mathcal{D}_P\oplus \mathcal{D}_P\oplus \dotsc$ takes the form 
		\begin{equation}\label{eqnX}
			X=  \begin{bmatrix} 
				A & 0 & 0 & 0 & \dotsc \\
				X_{21} & X_2 & 0 & 0 & \dotsc \\
				X_{31} & X_3 & X_2 & 0 & \dotsc \\
				X_{41} & X_4 & X_3 & X_2 & \dotsc \\
				\dotsc & \dotsc & \dotsc & \dotsc & \dotsc \\
			\end{bmatrix}\, ,
		\end{equation}
		where
		\begin{equation}\label{eqn6.2}
			X_{21}P+X_2D_P=D_PA \quad \text{and} \quad X_{n1}=X_{n+1, 1}P+X_{n+1}D_P, \quad n=2,3, \dotsc.
		\end{equation}
		Again, straightforward computations show that 
		\[
		XT_F=  \begin{bmatrix} 
			AS & 0 & 0 & 0 &  \dotsc \\
			X_{21}S+X_2F^*D_P & X_2F & 0 & 0 &  \dotsc \\
			X_{31}S+X_3F^*D_P & X_3F+X_2F^* & X_2F & 0 &  \dotsc \\
			X_{41}S+X_4F^*D_P & X_4F+X_3F^* & X_3F+X_2F^* & X_2F &  \dotsc \\
			X_{51}S+X_5F^*D_P & X_5F+X_4F^* & X_4F+X_3F^* & X_3F+X_2F^* &  \dotsc \\
			\dotsc & \dotsc & \dotsc & \dotsc & \dotsc \\
		\end{bmatrix}
		\]
		and
		\[ 
		T_FX=   \begin{bmatrix} 
			SA & 0 & 0 & 0 &  \dotsc \\
			F^*D_PA+FX_{21} & FX_2 & 0 & 0 &  \dotsc \\
			F^*X_{21}+FX_{31} & F^*X_2+FX_3 & FX_2 & 0 &  \dotsc \\
			F^*X_{31}+FX_{41} & F^*X_3+FX_4 & F^*X_2+FX_3 & FX_2  &  \dotsc \\
			F^*X_{41}+FX_{51} & F^*X_4+FX_5 & F^*X_3+FX_4 & F^*X_2+FX_3  &  \dotsc \\
			\dotsc & \dotsc & \dotsc & \dotsc & \dotsc \\
		\end{bmatrix}.
		\]
		Therefore, $X$ and $T_F$ commutes if and only if 
		\begin{equation}\label{eqn6.3}
			\left.\begin{split}	
			& (a) \quad X_{21}S+X_2F^*D_P=F^*D_PA+FX_{21}, \\
				& (b) \quad X_{n1}S+X_nF^*D_P=F^*X_{n-1, 1}+FX_{n1} \quad  \text{for} \quad  n=3,4,\dotsc, \\
				& (c) \quad X_2F=FX_2 \quad \text{and} \\
				& (d) \quad X_nF+X_{n-1}F^*=F^*X_{n-1}+FX_n \quad  \text{for} \quad  n=3,4,\dotsc.
			\end{split}\right\}
		\end{equation}
		Again, a sequence of routine computations yield
		\[
		X^*X=  \begin{bmatrix} 
			A^*A+\overset{\infty}{\underset{n=2}{\sum}}X_{n1}^*X_{n1} & \overset{\infty}{\underset{n=2}{\sum}}X_{n1}^*X_{n} & \overset{\infty}{\underset{n=2}{\sum}}X_{n+1, 1}^*X_{n} & \overset{\infty}{\underset{n=2}{\sum}}X_{n+2, 1}^*X_{n} & \dotsc \\
			\\
			\overset{\infty}{\underset{n=2}{\sum}}X_{n}^*X_{n1} & \overset{\infty}{\underset{n=2}{\sum}}X_{n}^*X_{n} & \overset{\infty}{\underset{n=2}{\sum}}X_{n+1}^*X_{n}  & 
			\overset{\infty}{\underset{n=2}{\sum}}X_{n+2}^*X_{n} & \dotsc \\
			\\
			\overset{\infty}{\underset{n=2}{\sum}}X_{n}^*X_{n+1, 1} & \overset{\infty}{\underset{n=2}{\sum}}X_{n}^*X_{n+1} & \overset{\infty}{\underset{n=2}{\sum}}X_{n}^*X_{n}  & 
			\overset{\infty}{\underset{n=2}{\sum}}X_{n+1}^*X_{n} & \dotsc \\
			\\
			\overset{\infty}{\underset{n=2}{\sum}}X_{n}^*X_{n+2, 1} & \overset{\infty}{\underset{n=2}{\sum}}X_{n}^*X_{n+2} & \overset{\infty}{\underset{n=2}{\sum}}X_{n}^*X_{n+1}  & 
			\overset{\infty}{\underset{n=2}{\sum}}X_{n}^*X_{n} & \dotsc \\
			
			\dotsc & \dotsc & \dotsc & \dotsc  & \dotsc \\
		\end{bmatrix}
		\]
		and
		\[
		T_F^*T_F=  \begin{bmatrix} 
			S^*S+D_PFF^*D_P & D_PF^2 & 0 & 0  & \dotsc \\
			F^{*2}D_{P} & F^*F+FF^* & F^2 & 0  & \dotsc \\
			0 & F^{*2} & F^*F+FF^* & F^2 & \dotsc\\
			0 & 0 & F^{*2} & F^*F+FF^* &  \dotsc\\
			\dotsc & \dotsc & \dotsc & \dotsc  & \dotsc \\
		\end{bmatrix}.
		\]
		Hence, $X^*X+\frac{1}{4}T_F^*T_F= I$ if and only if 
		\begin{equation}\label{eqn6.4}
			\left.\begin{split}
				& (a) \quad I-A^*A-\frac{1}{4}S^*S=\overset{\infty}{\underset{n=2}{\sum}}X_{n1}^*X_{n1}+\frac{1}{4}D_PFF^*D_P\\
				& (b) \quad
				\overset{\infty}{\underset{n=2}{\sum}}X_{n}^*X_{n}=I-\frac{1}{4}(F^*F+FF^*)\\
				& (c) \quad \overset{\infty}{\underset{n=2}{\sum}}X_{n}^*X_{n+k,1}=0=\overset{\infty}{\underset{n=2}{\sum}}X_{n+k+1}^*X_{n} \quad  \text{for} \ k=1,2, \dotsc\\
				& (d) \quad \overset{\infty}{\underset{n=2}{\sum}}X_{n1}^*X_{n}+\frac{1}{4}D_PF^2=0=\overset{\infty}{\underset{n=2}{\sum}}X_{n+1}^*X_{n}+\frac{1}{4}F^2.\\
			\end{split}\right\}
		\end{equation}
		Hence, the necessary part follows from equations (\ref{eqn6.2}) -- (\ref{eqn6.4}).\\
		
		Conversely, let us assume that the operator equations given in the statement of the theorem hold. Set 
		\[
		X=  \begin{bmatrix} 
			A & 0 & 0 & 0 & \dotsc \\
			X_{21} & X_2 & 0 & 0 & \dotsc \\
			X_{31} & X_3 & X_2 & 0 & \dotsc \\
			X_{41} & X_4 & X_3 & X_2 & \dotsc \\
			\dotsc & \dotsc & \dotsc & \dotsc & \dotsc \\
		\end{bmatrix}, \; T_F=  \begin{bmatrix} 
			S & 0 & 0 & 0 & \dotsc \\
			F^*D_{P} & F & 0 & 0 & \dotsc \\
			0 & F^* & F & 0 & \dotsc\\
			0 & 0 & F^* & F & \dotsc\\
			\dotsc & \dotsc & \dotsc & \dotsc & \dotsc \\
		\end{bmatrix},  \;
		V_0=  \begin{bmatrix} 
			P & 0 & 0 & 0 & \dotsc \\
			D_{P} & 0 & 0 & 0 & \dotsc \\
			0 & I & 0 & 0 & \dotsc\\
			0 & 0 & I & 0 & \dotsc\\
			\dotsc & \dotsc & \dotsc & \dotsc & \dotsc \\
		\end{bmatrix}
		\]
		on the space $\mathcal{H} \oplus \mathcal{D}_P\oplus \mathcal{D}_P\oplus \dotsc=\mathcal{H} \oplus \ell^2(\mathcal{D}_P)$. It follows from Theorem \ref{thm6.2} that $(T_F, V_0)$ is a $\Gamma$-isometry on $\mathcal{H} \oplus \ell^2(\mathcal{D}_P)$. Again, using the same computations for  equations (\ref{eqn6.2}) -- (\ref{eqn6.4}), we get that $X$ commutes with both $T_F$ as well as with $V_0$ and
	$
	X^*X+\frac{1}{4}T_F^*T_F=I.
	$
		Consequently, Theorem \ref{P_isometry} yields that $(X, T_F, V_0)$ is a $\mathbb{P}$-isometry on $\mathcal{H} \oplus \ell^2(\mathcal{D}_P)$. Evidently, 
		$
		A^*=X^*|_\mathcal{H}$, $S^*=T_F^*|_\mathcal{H}$ and $P^*=V_0^*|_\mathcal{H}$ and hence $(X,T_F,V_0)$ dilates $(A,S,P)$. The proof is now complete.		
	\end{proof}
	
	The conditions in Theorem \ref{thm:main-dilation} can be a bit relaxed if want a dilation of a special kind. Indeed, we shall see below that under seven of the conditions as in Theorem \ref{thm:main-dilation}, we can exhibit a particular $\Pe$-isometric dilation of a $\Pe$-contraction $(A,S,P)$ on the minimal dilation space of $P$. However, Theorem \ref{thm:main-dilation} provides the general case which can come only in the presence of all ten conditions.

	\begin{thm}\label{lem6.9}
		Let $(A, S, P)$ be a $\mathbb{P}$-contraction on a Hilbert space $\mathcal{H}$. If there are two operators $F_1, F_2\in \mathcal B(\mathcal{D}_P)$ satisfying the following:
		
		\begin{minipage}[t]{0.5\textwidth}
			\smallskip
			\begin{equation*}
				\begin{split}
					& 1. \ \ F_2D_PP+F_1D_P=D_PA, \\
					& 3.  \ \ F_1F=FF_1 , \\
					& 5.  \ \ F_2F+F_1F^*=FF_2+F^*F_1 , \\
					& 7. \ \ I-A^*A-\frac{1}{4}S^*S=D_P\bigg(F_2^*F_2+\frac{1}{4}FF^*\bigg)D_P ,\\		
				\end{split}
			\end{equation*}	
		\end{minipage}
		\begin{minipage}[t]{0.5\textwidth}
			\smallskip
			\begin{equation}\label{eqn7}
				\begin{split}
					& 2. \ \  F_2F^*=F^*F_2 , \\
					& 4.  \ \ F_2^*F_1+F^2 \slash 4=0, \\
					& 6. \ \ F_1^*F_1+F_2^*F_2=I-\frac{1}{4}(F^*F+FF^*),\\		
				\end{split}
			\end{equation}	
		\end{minipage}
		
		then $(X,T_F,V_0)$ on $\mathcal
	H \oplus \ell^2(\mathcal{D}_P)$ is a minimal $\Pe$-isometric dilation of $(A, S, P)$, where
		\[
		X=  \begin{bmatrix} 
		A & 0 & 0 & 0 & \dotsc \\
		F_2D_{P} & F_1 & 0 & 0 & \dotsc \\
		0 & F_2 & F_1 & 0 & \dotsc\\
		0 & 0 & F_2 & F_1 & \dotsc\\
		\dotsc & \dotsc & \dotsc & \dotsc & \dotsc \\
	\end{bmatrix}, \; T_F=  \begin{bmatrix} 
			S & 0 & 0 & 0 & \dotsc \\
			F^*D_{P} & F & 0 & 0 & \dotsc \\
			0 & F^* & F & 0 & \dotsc\\
			0 & 0 & F^* & F & \dotsc\\
			\dotsc & \dotsc & \dotsc & \dotsc & \dotsc \\
		\end{bmatrix},  \;
		V_0=  \begin{bmatrix} 
			P & 0 & 0 & 0 & \dotsc \\
			D_{P} & 0 & 0 & 0 & \dotsc \\
			0 & I & 0 & 0 & \dotsc\\
			0 & 0 & I & 0 & \dotsc\\
			\dotsc & \dotsc & \dotsc & \dotsc & \dotsc \\
		\end{bmatrix}.
		\]
	\end{thm}
	
	\smallskip

\begin{proof}

The minimality is immediate once we prove that $(X,T_F,V_0)$ is a $\Pe$-isometric dilation of $(A,S,P)$, because, $V_0$ acting on $\HS \oplus \ell^2(\mathcal D_P) $ is the minimal isometric dilation of $P$. If we put 
	\[
	X_2=F_1, \quad X_3=F_2, \quad X_{21}=F_2D_P \quad \text{and} \quad X_{n1}=0=X_{n+1} \quad \text{for} \ n \geq 3
	\]
	in Theorem \ref{thm:main-dilation}, then the conditions (1) and (9) in Theorem \ref{thm:main-dilation} become redundant and the operator $X$ takes the block-matrix form as in the statement of this theorem. Thus, to ensure that $(X, T_F,V_0)$ is a $\Pe$-isometric dilation of $(A,S,P)$ in view of Theorem \ref{thm:main-dilation}, it suffices to prove
	\[
F_2D_PS+F_1F^*D_P=F^*D_PA+FF_2D_P\,,
\]
because, the other conditions are the hypotheses of this theorem. We deduce the above condition from the identities $(1), (4)$ and $(5)$ in (\ref{eqn7}). Note that the fundamental operator $F$ of a $\Gamma$-contraction $(S, P)$ satisfies 
\begin{equation}\label{eqn6.7'}
	D_PS=FD_P+F^*D_PP.
\end{equation}
See the last section of \cite{Bhattacharyya} for a proof. Let $G=D_PS-FD_P-F^*D_PP$. Then $G:\mathcal{H} \to \mathcal{D}_P$ satisfies the following:
\begin{equation*}
	\begin{split}
		D_PG=D_P^2S-D_PFD_P-D_PF^*D_PP=(I-P^*P)S-(S-S^*P)-(S^*-P^*S)P=0.
	\end{split}
\end{equation*}
Now, $\langle Gx, D_Py\rangle=\langle D_PGx, y\rangle=0$ for all $x, y \in \mathcal{H}$, which implies that $G=0$. Now multiplying both sides of $F_2D_PP+F_1D_P=D_PA$ by $F^*$ both sides, we have
\begin{equation*}
	\begin{split}
		F^*D_PA& = F^*F_2D_PP+F^*F_1D_P\\
		&=F_2F^*D_PP+F^*F_1D_P \qquad [\text{ by condition-(2)}]\\
		&=F_2D_PS-F_2FD_P+F^*F_1D_P \quad [\text{ by } \ (\ref{eqn6.7'})] \\
		&=  F_2D_PS-(F_2F-F^*F_1)D_P   \\
		&=F_2D_PS-(FF_2-F_1F^*)D_P. \quad [\text{ by condition-(5) of } (\ref{eqn7})]\\	
	\end{split}
\end{equation*}
Thus, we have that $F_2D_PS+F_1F^*D_P=F^*D_PA+FF_2D_P$ and this completes the proof.

\end{proof}

\begin{rem}\label{rem:last-section}
The conditional dilations as in Theorems \ref{thm:main-dilation} \& \ref{lem6.9} determine a class of $\Pe$-contractions $(A,S,P)$ that dilate to $\Pe$-isometries on the minimal isometric dilation space for $P$. However, there are limitations these theorems mainly because the concerned dilation space, i.e. $\HS \oplus \ell^2(\mathcal D_P)$ is too small. Below we provide examples to show that Theorems \ref{thm:main-dilation} \& \ref{lem6.9} provide dilations to nontrivial classes of $\Pe$-contractions and also at the same time they are not applicable for some $\Pe$-contractions.

\begin{enumerate}

\item Let $T$ be a contraction acting on a Hilbert space $\mathcal{H}$ such that $D_TT=0$. Then the $\mathbb{P}$-contraction $(A, S, P)=(T, 0, 0)$ admits a $\mathbb{P}$-isometric dilation $(X, T_F, V)$, where $V$ is a minimal isometric dilation space of $P$. Indeed, it follows from Theorem \ref{lem6.9} that if there exist $F_1$ and $F_2$ in $\mathcal{B}(\mathcal{D}_P)$ such that the operator equations in (\ref{eqn7}) hold, then the desired conclusion follows. Here $F=0, D_P=I$ and so $\mathcal{D}_P=\mathcal{H}$. A straightforward computation shows that for $(F_1, F_2)=(T, D_T)$, the operator equations in (\ref{eqn7}) admit a solution.

\smallskip

\item Proposition \ref{prop2.8} yields that $(I, 0, T)$ is a $\Pe$-contraction for any contraction $T$. Then the $\mathbb{P}$-contraction $(A, S, P)=(I, 0, T)$ admits a $\mathbb{P}$-isometric dilation $(X, T_F, V)$, where $V$ is a minimal isometric dilation space of $P$. The choice of $F_1=I$ and $F_2=0$ in $\mathcal{B}(\mathcal{D}_P)$ gives a solution to (\ref{eqn7}) and the rest follows from Theorem \ref{lem6.9}.

\smallskip

\item On the other hand, $(A,S,P)=(0,0,I)$ on $\mathbb{C}^2$ is a $\Pe$-contraction as it is a commuting normal triple with $\sigma_T(A,S,P)=\{ (0,0,1) \} \subset \overline{\Pe}$. Now, since $(S, P) =(0,I)$ on $\mathbb{C}^2$ is a $\Gamma$-isometry (in fact a $\Gamma$-unitary), it follows that the minimal isometric dilation space of $P$ is $\mathbb{C}^2$ itself. Note that $(0,0,I)$ is not a $\Pe$-isometry as the first two component i.e., $(0,0)$ is not a $\mathbb B_2$-isometry. If $(0, 0, I)$ were to dilate to a $\mathbb{P}$-isometry $(X,T,V)$ on $\mathbb{C}^2$ with $V$ being the minimal isometric dilation of $I$, then $P_{\C^2}X|_{\C^2}=X=0$, $P_{\C^2}T|_{\C^2}=T=0$ and $P_{\C^2}V|_{\C^2}=V=I$, but Theorem \ref{P_unitary} yields that $X^*X+\frac{1}{4}T^*T=I$, which is a contradiction. Hence, $(0,0,I)$ does not dilate to a $\Pe$-isometry on the minimal isometric dilation space of the last component.
\end{enumerate} 
\end{rem}

If we move out of the territory of the minimal isometric dilation of the last component as in Theorems \ref{thm:main-dilation} \& \ref{lem6.9}, we can find $\Pe$-isometric dilation for some $\Pe$-contractions as shown below.

	\begin{prop}
		Every $\Pe$-contraction of the form $(T_1, 0, T_2)$ admits a $\Pe$-isometric dilation. \end{prop}
	
	\begin{proof}

It follows from Proposition \ref{prop2.8} that $(T_1, T_2)$ is a commuting pair of contractions if and only if $(T_1,0, T_2)$ is a $\Pe$-contraction.	
	 Let $(T_1, 0, T_2)$ be a $\Pe$-contraction acting on a Hilbert space $\mathcal{H}$.  A famous result due to Ando (see Chapter-I of \cite{Nagy}) yields that $(T_1,T_2)$ dilates to a pair of commuting isometries $(V_1, V_2)$. By Corollary \ref{V_1, 0, V_2}, we have that $(V_1, 0, V_2)$ is a $\Pe$-isometry. Evidently, $(V_1, 0, V_2)$ is a $\Pe$-isometric dilation of $(T_1, 0, T_2)$.
	\end{proof}
	
	\vspace{0.2cm}
	
A major role in the $\Pe$-isometric dilation of Theorem \ref{lem6.9} is played by the existence of a solution to the operator equation
	\begin{equation}\label{eqn:dilation-01}
	I-A^*A-\frac{1}{4}S^*S=D_P\big(Z^*Z+\frac{1}{4}FF^*\big)D_P.
	\end{equation}
 Indeed, it is evident from Theorem \ref{lem6.9}	that if (\ref{eqn:dilation-01}) has a solution $Z=F_2$ satisfying $I-\dfrac{1}{4}(F^*F+FF^*)-F_2^*F_2 \geq 0$, then it confirms the existence of $F_1$ and the rest boils down to $F_1,F_2$ satisfying the other identities. For this reason, we put special emphasis on (\ref{eqn:dilation-01}). In other words, we seek a solution $X \in \mathcal (\mathcal{D}_P)$ to the operator equation
	\begin{equation}\label{eqn6.8}
		I-A^*A-\frac{1}{4}S^*S=D_PXD_P.
	\end{equation}
	such that 
	$
	X=F_2^*F_2+\frac{1}{4}FF^*.
	$
	 Moreover, if there is a solution to (\ref{eqn6.8}), then $X\geq 0$ and consequently $D_PXD_P \geq 0$ which implies that $I-A^*A-\frac{1}{4}S^*S\geq 0$. Then we have
	\[
	\langle XD_Ph, D_Ph \rangle = \langle D_PXD_Ph, h \rangle= \langle (I-A^*A-S^*S \slash 4)h,h \rangle \leq \langle h , h\rangle=\|h\|^2
	\]
for every $h \in \mathcal{H}$ and hence it is necessary that $\omega(X) \leq 1$. In this connection let us recall an important result associated with the numerical radius.
	\begin{lem}[\cite{Bhattacharyya}, Lemma 2.9]
		The numerical radius of an operator $T$ is not greater than one if and only if Re $\beta T \leq I$ for all complex numbers $\beta$ of unit modulus. 
	\end{lem}
	
	It follows from the above lemma that $\omega\left( Z^*Z+\frac{1}{4}FF^* \right) \leq 1$ if and only if $Re \ \beta\left(Z^*Z+\frac{1}{4}FF^*\right) \leq I$ for all $\beta \in \T$. This is equivalent to saying that $Z^*Z+\frac{1}{4}FF^* \leq I$ as $Z^*Z+\frac{1}{4}FF^*$ is self-adjoint. In order to solve the operator equations in Theorem \ref{lem6.9}, we must have 
	\[
	I-F_2^*F_2-\frac{1}{4}FF^*=F_1^*F_1+\frac{1}{4}F^*F \geq 0.
	\]
	Thus, to obtain solutions that fit in with the system of equations in Theorem \ref{lem6.9}, we have to find $X \geq 0$ in $\mathcal{B}(\mathcal{D}_P)$ with $\omega(X) \leq 1$ such that (\ref{eqn6.8}) is satisfied. Our next result characterizes the class of $\Pe$-contractions for which $(\ref{eqn6.8})$ has a solution with the desired properties. The proof of this result requires the following lemma.

	\begin{lem}[\cite{Bhattacharyya}, Lemma 4.1] \label{lem6.11}
		Let $\Sigma$ and $D$ be two bounded operators on a Hilbert space $\mathcal{H}$. Then 
		\[
		DD^* \geq Re (e^{i\theta}\Sigma) \quad \text{for all} \ \theta \in \mathbb{R}
		\]
		if and only if there is $X \in \mathcal{B}(\mathcal{D}_*)$ with numerical radius of $X$ not greater than $1$ such that $\Sigma=DXD^*$, where $\mathcal{D}_*=\overline{Ran}(D^*)$.	
	\end{lem}

	\begin{thm}\label{thm6.17}
		Let $(A, S, P)$ be a $\Pe$-contraction on a Hilbert space $\mathcal{H}$. Then there is a unique solution $X \in \mathcal{B}(\mathcal{D}_P)$ with $\omega(X) \leq 1$ to the operator equation (\ref{eqn6.8}), i.e. $I-A^*A-\frac{1}{4}S^*S=D_PXD_P$ if and only if 
		\begin{equation}\label{eqn6.10}
			\pm	\left(I-A^*A-\frac{1}{4}S^*S\right) \leq  D_P^2.
		\end{equation}
		Moreover, if such a solution $X$ exists, then
		$
		X \geq 0 \ \text{if and only if} \ \text{$(A, S\slash 2)$ is a spherical contraction}.
		$ 
	\end{thm}

	\begin{proof}
		Let 
		$
		\Sigma=I-A^*A-\frac{1}{4}S^*S$ and $D=D_P$. Then, it follows from Lemma \ref{lem6.11} that there is an operator $X \in \mathcal{B}(\mathcal{D}_P)$ with $\omega(X) \leq 1$ such that 
		$
		I-A^*A-\frac{1}{4}S^*S=D_PXD_P
		$	
		if and only if 
		\begin{equation*}
			\begin{split}
				0 & \leq D_P^2-Re(e^{i\theta} \Sigma)=D_P^2-\Sigma \ Re(e^{i\theta})=D_P^2-\cos\theta \ \Sigma \quad \text{for all} \ \theta \in \mathbb{R}. 
			\end{split}
		\end{equation*}
		We show that 
		$
		D_P^2 \geq \cos \theta \Sigma $ for all $\theta \in \mathbb{R}$ if and only if $ D_P^2 \geq \pm \ \Sigma$. The necessary part is obvious. We prove the converse. Let $D_P^2 \geq \pm \ \Sigma$. Since $\Sigma$ is a self-adjoint operator, we have that
		$
		\langle \Sigma x, x \rangle \in \mathbb{R}$ for every $x \in \mathcal{H}$.
		Take any $\theta \in \mathbb{R}$ and $x \in \mathcal{H}$. We consider two different cases here depending on whether $\langle \Sigma x, x \rangle$ is positive or negative. If $\langle \Sigma x, x \rangle \geq 0$, then
		$
		\cos \theta \langle \Sigma x, x \rangle \leq \langle \Sigma x, x \rangle \leq \langle D_P^2x, x \rangle.
		$
		Also, if $\langle \Sigma x, x \rangle \leq 0$, then
		$
		\cos \theta \langle \Sigma x, x \rangle \leq -\langle \Sigma x, x \rangle \leq \langle D_P^2x, x \rangle.
		$
		In either case, we have that $\langle (D_P^2-\cos \theta \Sigma) x, x \rangle \geq 0$. Thus, $D_P^2 \geq \cos\theta \ \Sigma$ for all $\theta \in \mathbb{R}$. Thus, there is a solution $X \in \mathcal{B}(\mathcal{D}_P)$ with $\omega(X) \leq 1$ to the operator equation 
	$
		I-A^*A-\frac{1}{4}S^*S=D_PXD_P
		$
		if and only if $D_P^2 \geq \pm\Sigma$ which is equivalent to saying that 
		$
		\pm	\left(I-A^*A-\frac{1}{4}S^*S\right) \leq  D_P^2$.
		
		\smallskip
		
		For the uniqueness part, let there be two such solutions $X_1$ and $X_2$. Then 
		$
		D_P\widehat{X}D_P=0$, where $ \widehat{X}=X_1-X_2 \in \mathcal{B}(\mathcal{D}_P)$. Then, for all $x, y \in \mathcal{H}$, we have
		$
		\langle \widehat{X}D_Px, D_Py\rangle=\langle D_P\widehat{X}D_Px, y \rangle =0 ,
		$
		which shows that $\widehat{X}=0$. Hence, $X_1=X_2$.

\smallskip

Let us assume that there is an operator $X \in \mathcal{B}(\mathcal{D}_P)$ such that (\ref{eqn6.8}) holds. For any $x \in \mathcal{H}$, we have 
		\[
		\langle XD_Px, D_Px \rangle=\langle D_PXD_Px, x \rangle=\left\langle \left(I-A^*A-\frac{1}{4}S^*S\right)x, x \right \rangle
		\]
		which shows that 
		$
		I-A^*A-\frac{1}{4}S^*S \geq 0$ if and only if $ X \geq 0$.	The proof is now complete.
	\end{proof}
	
	Note that (\ref{eqn6.10}) does not hold for all $\Pe$-contractions. The scalar version of (\ref{eqn6.10}) is given by
			\[
			\pm \left(1-|a|^2-\frac{1}{4}|s|^2\right) \leq 1-|p|^2.
			\]  
			Now $(a, s, p)=(0, 0, 1)$, which in $\PC$, does not satisfy the above inequality. Also, $I-A^*A-\frac{1}{4}S^*S \geq 0$ if and only if $(A,S \slash 2)$ is a spherical contraction. Again, for every $\Pe$-contraction $(A,S,P)$ we have that $(A, S \slash 2)$ is a $\B_2$-contraction. Thus, we are in search of $\B_2$-contractions that are spherical contractions. In the special case when $(A,S,P)$ is a subnormal $\Pe$-contraction, i.e., a $\Pe$-contraction that admits an extension to a normal $\Pe$-contraction, we have success by an application of an elegant result due to Athavale, \cite{AthavaleII}.

	\begin{lem}\label{lem7.20}
		A subnormal $\Pe$-contraction $(A, S, P)$ satisfies $I-A^*A-\frac{1}{4}S^*S \geq 0$. 
	\end{lem}
	
	\begin{proof}
		Let $(A, S, P)$ be a subnormal $\Pe$-contraction. It follows from Proposition \ref{prop2.11} that $\sigma_T(A, S\slash 2) \subseteq \BC$.  Now Theorem 5.2 in \cite{AthavaleII} yields that 
		$
		I-A^*A-\frac{1}{4}S^*S \geq 0.
		$
	\end{proof}

	\noindent It is never easy to determine the success or failure of rational dilation on a domain. Rational dilation succeeds on the bidisc $\mathbb D^2$ and on the symmetrized bidisc $\mathbb G_2$ (see \cite{AglerII, Bhattacharyya}), but it is unclear at this point if it succeeds on the pentablock. No domain in $\C^n$ for $n>2$ is known to have an affirmative answer for the rational dilation problem. Thus, our wild guess is that rational dilation fails on the pentablock. Our future plan is to investigate an answer to this problem for the pentablock via operator theory on the biball $\B_2$.
	
\medskip	

\section{Data Availability Statement}

\noindent (1) Data sharing is not applicable to this article, because, as per our knowledge no datasets were generated or analysed during the current study.

\smallskip

\noindent (2) In case any datasets are generated and/or analysed during the current study which go unnoticed, they must be available from the corresponding author on reasonable request.

\section{Declarations}

\noindent \textbf{Ethical Approval.} This declaration is not applicable.

\smallskip

\noindent \textbf{Competing interests.} There are no competing interests.

\smallskip

\noindent \textbf{Authors' contributions.} All authors have contributed equally.

\vspace{0.2cm}

\noindent \textbf{Funding.} The first named author is supported by ``Core Research Grant” of Science and Engineering Research Board (SERB), Govt. of India, with Grant No. CRG/2023/005223 and the ``Early Research Achiever Award Grant” of IIT Bombay with Grant No. RI/0220-10001427-001.  The second named author is supported by the Prime Minister’s Research Fellowship (PMRF) with PMRF Id No. 1300140 of Govt. of India.

\end{document}